\documentclass[12pt]{article}
\usepackage{amsmath,amssymb,amsthm,amscd}
\usepackage[all]{xy}
\usepackage{graphicx}
\usepackage{tikz}

\newcommand{\N}{\mathbb{N}}
\newcommand{\Z}{\mathbb{Z}}

\newcommand{\R}{\mathbb{R}}
\newcommand{\C}{\mathbb{C}}
\newcommand{\T}{\mathbb{T}}

\newtheorem{defn}{Definition}[section]
\newtheorem{thm}[defn]{Theorem}
\newtheorem{ex}[defn]{Example}
\newtheorem{prop}[defn]{Proposition}
\newtheorem{cor}[defn]{Corollary}
\newtheorem{lemma}[defn]{Lemma}
\newtheorem{ques}[defn]{Question}
\newtheorem{rem}[defn]{Remark}

\title{Binary factors of shifts of finite type}

\author{
 Ian F. Putnam\thanks{Supported in part by a
Discovery Grant from NSERC, Canada},\\
Department of Mathematics and Statistics,\\
University of Victoria,\\
Victoria, B.C., Canada V8W 3R4}

\begin{document}

\maketitle

\begin{abstract}
We construct two new classes of topological 
dynamical systems; one is a factor
of a one-sided shift of finite type while the second is 
a factor of the two-sided shift. The data is a finite graph 
which presents the shift of finite type, a second finite directed
graph and a pair of embeddings of it into the first, 
satisfying certain conditions. The factor is then obtained
from a simple idea based on binary expansion of real numbers.
In both cases, we construct natural metrics on the factors 
and, in the second case, this makes the system a Smale space, 
in the sense of Ruelle. We compute various algebraic invariants
for these systems, including the homology for Smale space
 developed by the author and the K-theory of various $C^{*}$-algebras
 associated to them, 
in terms of the pair of original graphs.
\end{abstract}

\section{Introduction}
\label{intro}
In the subject of topological dynamical systems, a crucial part has been 
played by systems which display some type of hyperbolicity; that is, systems 
which have some type of local expanding or contracting/expanding behaviour.
This began in the smooth category with the study of Anosov
diffeomorphisms \cite{KH:book}. 
Smale's seminal paper \cite{Sm:Diff} showed that,
even for smooth systems, the hyperbolic behaviour may be limited  
to the non-wandering set which may be far from being a submanifold.
Motivated by this, David Ruelle gave a purely topological definition
of hyperbolicity which he called a Smale space \cite{Ru:ThFor}. This broader
framework includes shifts of finite type, which are highly 
combinatorial in nature (see section \ref{prelim}
for the definition). Restricted to irreducible systems, 
these are precisely the 
 Smale 
spaces whose underlying space is totally disconnected.

One of the fundamental 
 ideas for  these systems is that they may be
coded by a Markov partition \cite{KH:book}. Ignoring some minor 
technical issues, the existence of a Markov 
 is equivalent to the condition that there is an almost-one-to-one 
  factor
  map from a shift of finite type onto the given system.
  As noted by Adler in \cite{Ad:MP},
  the simplest example of this is binary 
   expansion of real numbers: the space of one-sided $0,1$-sequences
   maps onto the unit circle via the familiar formula
  sending  a sequence $x_{n}, n \geq 1,$ to 
  $\exp ( 2 \pi i \sum_{n} x_{n}2^{-n} )$. Moreover, this map 
  intertwines the dynamics of the left shift map on the 
  sequences with the squaring map on the circle. ('Binary' may be 
  changed to
  'decimal' by replacing certain $2$'s with $10$'s.)
  One tends to regard decimal expansion as a bijection, but 
  the two spaces here are very different topologically.
  
  Returning to the general situation, we recall the following 
  seminal result. In the form stated, it is due to Rufus Bowen \cite{Bo:MP}, 
  but it builds on the work of many others, including Sinai \cite{Si:MP},
   Adler and Weiss \cite{AW:MP} and others. Let $(X, d)$ be a compact metric 
   space and $\varphi$ be a homeomorphism of $X$ such that 
   $(X, \varphi, d)$ is a irreducible 
   Smale space (see section \ref{smale}
   for the definition). Alternately, let $(X, \varphi)$ be 
   the restriction of an Axiom A system to a basic set.
   Then there is a irreducible shift of finite type
   $(\Sigma, \sigma)$ and a continuous surjection
    $\pi: \Sigma \rightarrow X$ such that 
      $\pi \circ \varphi = \sigma \circ \pi$.
      
   There are several interesting invariants of  a 
  Smale space. First, one can construct a number of different
  $C^{*}$-algebras from a single Smale space. 
  For any Smale space, $(X, d, \varphi)$ and choice of a finite 
 $\varphi$-invariant set $P \subseteq X$, there are $C^{*}$-algebras
 $S(X, \varphi, P)$ and $U(X, \varphi, P)$ based on stable and unstable 
 equivalence, respectively. Each has a canonical isomorphism induced 
 by $\varphi$ and we let
  $R^{s}(X, \varphi, P ) = S(X, \varphi, P) \rtimes_{\varphi} \Z$
 and $R^{u}(X, \varphi, P)  = U(X, \varphi, P) \rtimes_{\varphi} \Z$
 be the associated crossed product $C^{*}$-algebras (see section 
 \ref{c*}).
 
 This was 
  done initially for shifts of finite type by Cuntz and Krieger
  \cite{CK:alg} and later by Ruelle \cite{Ru:AlgHypDiff} for general Smale spaces.
  The K-theory groups of these $C^{*}$-algebras have been 
  investigated quite thoroughly and provide 
  interesting data. For shifts of finite type, these include
  Krieger's dimension group invariant as well as 
  the Bowen-Franks groups \cite{BF:hom}. There are, in fact, two dimension groups
  which are associated with right and left tail-equivalence, respectively.
  These are described  in section 
  \ref{hom} but we note that for a shift of finite type
  associated with a finite directed graph $G$, these
   have simple combinatorial
  descriptions in terms of $G$ and we denote  them by $D^{s}(G)$ and $D^{u}(G)$.
  We let $A_{G}$ denote the adjacency matrix associated with $G$. The matrices
  $A_{G}^{T}$ and $A_{G}$ induce automorphisms of 
  $D^{s}(G)$ and $D^{u}(G)$, respectively.

  In another direction, the author  constructed a homology theory
  for Smale spaces \cite{Pu:HSm} which generalizes Krieger's dimension group
  for shifts of finite type. Indeed, the dimension group is 
  a key part of the construction. The existence of this theory, 
  which provides a Lefschetz formula, was conjectured by Bowen \cite{Bo:CBMS}.
  For any Smale space $(X, d, \varphi)$, 
   there are groups $H^{s}_{N}(X, \varphi)$, 
   $ H^{u}_{N}(X, \varphi),
    N \in \Z$, 
   and each has a canonical automorphism induced by $\varphi$.
  
  It is worth noting that recent work of 
  Proietti and Yamashita \cite{PY:hom1,PY:hom2} shows that there 
  are very close relations between the K-theory of
   the $C^{*}$-algebras and the Smale space homology, as well the 
   cohomology of certain groupoids.

  The main goal of this paper is a kind of 
  reverse-engineering of Bowen's result, where the 
  Smale space is the end product of the construction, rather
  than the initial object of interest. We begin with a 
some combinatorial data: a pair of  
  shifts of finite
 type associated with finite directed graphs $G$ and $H$ with a pair of 
 embeddings of the latter into the former
  $\xi^{0}, \xi^{1}: H \rightarrow G$.  We let $A_{G}$ and $A_{H}$ 
  denote the adjacency matrices for the two graphs. 
  We defer the details of the construction as well 
  as certain hypotheses on our data  to section \ref{constr} and
  focus on the properties of the resulting systems. In 
  particular, the 
  \emph{standing hypotheses} are described in Definition 
  \ref{constr:10}.
   
  As described in detail in section \ref{prelim}, we 
  let $(X^{+}_{G}, \sigma)$ and $(X_{G}, \sigma)$ denote 
  the one and two-sided shifts, respectively, 
   associated with the graph $G$.
 Adding a basic idea modeled on
  binary expansion, we construct topological dynamical systems
  $(X_{\xi}^{+}, \sigma_{\xi})$ and 
  $(X_{\xi}, \sigma_{\xi})$ along with factor maps
  $\pi_{\xi}: (X^{+}_{G}, \sigma) \rightarrow (X_{\xi}^{+}, \sigma_{\xi})$
  and 
   $\pi_{\xi}: (X_{G}, \sigma) \rightarrow (X^{+}, \sigma_{\xi})$.
   Both systems are continuous and surjective  and the second 
   is actually a homeomorphism.
   We also construct specific metrics $d_{\xi}$ on our two spaces
   having nice properties. Indeed, most of the effort lies in 
   producing the metrics.
   
 Let us summarize the main results on these systems.
 The first is a local expansiveness property for the first system. 
  
  \newtheorem*{thm10}{Theorem~\ref{constr:190}}
  \begin{thm10}
  If $x,y$ in $X^{+}_{\xi}$ satisfy $d_{\xi}(x,y) \leq 2^{-2}$, then 
  \[
2 d_{\xi}(x,y) \leq   d_{\xi}(\sigma_{\xi}(x), \sigma_{\xi}(y))
 \leq  8 d_{\xi}(x,y).
  \]
  \end{thm10}

  In fact, the map is actually a local homeomorphism.

   \newtheorem*{thm20}{Theorem~\ref{constr:210}}
   \begin{thm20}
  If $x, y$ are in $X_{\xi}^{+}$ with $d_{\xi}(x, \sigma_{\xi}(y)) \leq 2^{-1}$, 
  then there is $z$ in $X_{\xi}^{+}$ with 
  $d_{\xi}(z,y) \leq 2^{-1} d_{\xi}(x, \sigma_{\xi}(y))$ and $\sigma_{\xi}(z) = x$.
  In particular,
  we
  have 
  \[
  X_{\xi}^{+}(\sigma_{\xi}(y), \epsilon) \subseteq 
  \sigma_{\xi}(X_{\xi}^{+}(y, 2^{-1} \epsilon)),
  \]
  for any $\epsilon < 2^{-1}$
 and the map $\sigma_{\xi}$ is open.
  \end{thm20}
  
  The actual geometric structure of $X_{\xi}^{+}$ is rather curious. 
  We offer an intriguing picture of a single example in section 
  \ref{real},but we also note the following.
  
   \newtheorem*{cor30}{Corollary~\ref{real:80}}
  \begin{cor30}
 The connected subsets of $X_{\xi}^{+}$ are 
either points or circles and both occur.
 \end{cor30}

  The  property analogous to locally expanding 
  for the second system is that it
  possess local coordinates of contracting and expanding directions.
  In short, it is a Smale space. (We review the definition in 
  section \ref{smale}.) 
  
  \newtheorem*{thm40}{Theorem~\ref{smale:30}}   
\begin{thm40}
 $(X_{\xi}, d_{\xi}, \sigma_{\xi})$ is a Smale space.
\end{thm40}

We also note the following (abridged version). Here, $X^{s}(x)$ denotes 
the global stable set of the point $x$ in a Smale space $X$.

 \newtheorem*{thm50}{Theorem~\ref{smale:50}}
\begin{thm50}  
 The map $\pi_{\xi}: (X_{G}, \sigma) \rightarrow (X_{\xi}, \sigma_{\xi})$
  is $s$-bijective; that is, for every $x$ 
in $X_{G}$, $\pi_{\xi}| X^{s}_{G}(x)$ is a bijection
from  $  X^{s}_{G}(x)$  to  $  X^{s}_{\xi}(\pi_{\xi}(x))$.
\end{thm50}

This means, in particular, that the local stable sets of $X_{\xi}$ are Cantor sets.

 The following summarizes our computation of the homology theory.  
   
   \newtheorem*{thm60}{Theorem~\ref{hom:20}} 
 \begin{thm60}
 Under the standing hypotheses, we have 
 \[
 \begin{array}{cccccc}
 H^{s}_{0}(X_{\xi}, \sigma_{\xi}) & \cong & D^{s}(G), &
 (\sigma_{\xi})_{*}^{-1} & \cong & A_{G}, \\
  H^{s}_{1}(X_{\xi}, \sigma_{\xi}) & \cong & D^{s}(H), &
 (\sigma_{\xi})_{*}^{-1} & \cong & A_{H},\\
   H^{u}_{0}(X_{\xi}, \sigma_{\xi}) & \cong & D^{u}(G),&
 (\sigma_{\xi})_{*} & \cong & A_{G}^{T}, \\
  H^{u}_{1}(X_{\xi}, \sigma_{\xi}) & \cong & D^{u}(H), &
 (\sigma_{\xi})_{*} & \cong & A_{H}^{T},\\
   H^{s}_{k}(X_{\xi}, \sigma_{\xi}) & \cong & 0, & k & \neq &  0,1,   \\
  H^{u}_{k}(X_{\xi}, \sigma_{\xi}) & \cong & 0,  & k &  \neq & 0,1.  
  \end{array}
  \]
  In  the first two lines, we regard $\sigma_{\xi}$ as an $s$-bijective
  factor map from $(X_{\xi}, \sigma_{\xi})$ to itself and 
  our description of the induced map on homology
  is interpreted via the isomorphism which precedes it. In the next two lines, 
  we regard $\sigma_{\xi}$ as a $u$-bijective
  factor map. 
 \end{thm60}

The next two results summarize our computations for the K-theory of
the $C^{*}$-algebras.
 
  \newtheorem*{thm70}{Theorem~\ref{c*:40}}
\begin{thm70}
Under the standing hypotheses, 
 we have
\begin{enumerate}
\item 
$K_{0}(S(X_{\xi}, \sigma_{\xi}, P_{\xi})) \cong D^{u}(G)$
as ordered abelian groups
and, under this isomorphism,  the automorphism induced by $\sigma_{\xi}$
is $A_{G}^{T}$.
\item 
$K_{1}(S(X_{\xi}, \sigma_{\xi}, P_{\xi})) \cong D^{u}(H)$
and, under this isomorphism,  the automorphism induced by $\sigma_{\xi}$
is $A_{H}^{T}$.
\item $K_{0}(R^{s}(X_{\xi}, \sigma_{\xi}, P_{\xi})) \cong \Z^{G^{0}}/ (I-A_{G}^{T})\Z^{G^{0}}
\oplus \ker( I - A_{H}^{T}: \Z^{H^{0}}  \rightarrow  \Z^{H^{0}} ).$
\item $K_{1}(R^{s}(X_{\xi}, \sigma_{\xi}, P_{\xi})) \cong \Z^{H^{0}}/ (I-A_{H}^{T})\Z^{H^{0}}
\oplus \ker( I - A_{G}^{T}: \Z^{G^{0}}  \rightarrow  \Z^{G^{0}} ).$
\end{enumerate}
\end{thm70}

\newtheorem*{thm80}{Theorem~\ref{c*:50}}
\begin{thm80}
Under the standing hypotheses, 
 we have
\begin{enumerate}
\item 
$K_{0}(U(X_{\xi}, \sigma_{\xi}, P_{\xi})) \cong D^{s}(G)$
as ordered abelian groups
and, under this isomorphism,  the automorphism induced by $\sigma_{\xi}$
is $A_{G}^{-1}$.
\item 
$K_{1}(U(X_{\xi}, \sigma_{\xi}, P_{\xi})) \cong D^{s}(H)$
and, under this isomorphism,  the automorphism induced by $\sigma_{\xi}$
is $A_{H}^{-1}$.
\item $K_{0}(R^{u}(X_{\xi}, \sigma_{\xi}, P_{\xi})) \cong \Z^{G^{0}}/ (I-A_{G})\Z^{G^{0}}
\oplus \ker( I - A_{H}: \Z^{H^{0}}  \rightarrow  \Z^{H^{0}} ).$
\item $K_{1}(R^{u}(X_{\xi}, \sigma_{\xi}, P_{\xi})) \cong \Z^{H^{0}}/ (I-A_{H})\Z^{H^{0}}
\oplus \ker( I - A_{G}: \Z^{G^{0}}  \rightarrow  \Z^{G^{0}} ).$
\end{enumerate}
\end{thm80} 
   
   A curious consequence of these computations and 
   the Phillips-Kirchberg classification theorem for $C^{*}$-algebras
   is the following.

  \newtheorem*{cor90}{Theorem~\ref{c*:80}} 
   \begin{cor90}
Let $(X, \varphi, d)$ be a mixing Smale space and $Q$ be
 a finite $\varphi$-invariant subset of $X$. There exist finite directed graphs,
$G, H$ and embeddings $\xi^{0}, \xi^{1}: H \rightarrow G$ satisfying 
the standing hypotheses such that \newline
$R^{s}(X_{\xi}, \sigma_{\xi}, P_{\xi}) \cong R^{s}(X, \varphi, Q)$,
for any $P$, a finite $\sigma$-invariant subset of $X_{G}$.
\end{cor90}

   The second section contains basic background information
   on shifts of finite type, or at least edge shifts
   for finite directed graphs.
   The third section is the construction of our factor system.
   At this stage, we begin only with one-sided shifts of finite type
   and construct factors which are continuous topological 
   dynamical systems, but  not invertible. Much of the effort 
  here is focused on the construction of a specific metric with 
  nice properties. In the fourth section, we continue the construction
  to produce invertible systems. This is a standard technique by 
  inverse limits,  but again, we construct a specific metric which makes
  the result a Smale space. We also review the definitions of a Smale space.
  The fifth section is devoted to the computation of the homology theory
  for our Smale spaces. The sixth deals with the various 
  $C^{*}$-algebras associated with our examples, including computations 
  of their K-theories. In the seventh section, we return to the factors 
  of the one-sided shifts of finite type. Geometrically, shifts of finite
  type are not particularly interesting; at least they are no more interesting
  the Cantor ternary set. This not is true of our factors and we
  spend some time giving a description of their geometry. 
  In particular, we show that the spaces can be embedded in $\R^{3}$. 
  We also give a couple of simple examples, which
   can actually be embedded in the plane. 
   
   It is a pleasure to thank Michael Barnsley 
   for helpful conversations. I am particularly indebted  to Mitch 
   Haslehurst for many interesting conversations on these matters, but 
   especially for initially drawing my attention to his pictures 
   for the example in section 7 (the ones I give 
   are slightly different) which piqued my curiosity.


\section{Preliminaries}
\label{prelim}
In this section, we set out some well-known preliminaries
on shift spaces. An excellent reference is the book of
Lind and Marcus \cite{LM:book}. As discussed in the introduction, we
are considering  shifts of finite type. However, we will deal 
exclusively with edge shifts of finite directed graphs. These 
are dynamically equivalent, as a consequence of
 Example 1.5.10 and  Theorem 2.3.2 of \cite{LM:book}.

By a \emph{finite directed graph}, $G$, we mean two 
finite sets $G^{0}$ (the vertex set)
and $G^{1}$ (the edge set) along with maps $i,t: G^{1} \rightarrow G^{0}$.
 A path of length $n$ in $G$ is a finite sequence 
 $p = (p_{1}, p_{2}, \ldots, p_{n})$ 
 in $G^{1}$ such that $t(p_{i}) = i(p_{i+1})$, for all $ 1 \leq i < n$.
 We define $i(p) = i(p_{1})$ and $t(p)=t(p_{n})$. 
 
The \emph{adjacency matrix} for the graph is denoted $A_{G}$ and we  regard it as
the function on $G^{0} \times G^{0}$ whose value at $(v,w)$ is the number of edges
$e$ with $i(e) =w$ and $t(e) = v$. We also regard it an endomorphism 
of $\Z^{G^{0}}$ by matrix multiplication.

We say $G$ is \emph{irreducible} if, for every ordered pair of vertices $v, w$, there 
is a path $p$ with $i(p)=v, t(p)=w$. Equivalently, $G$ is irreducible, if for every
$(v,w)$, there is a positive integer $k$ with $A_{G}^{k}(w,v)$ positive. We say 
$G$ is \emph{primitive}
 if there is a positive integer $k$ such that 
$A_{G}^{k}(w,v)$ is positive, for all $v,w$.

We let $X_{G}^{+}$ be the one-sided infinite path space: an 
element is a sequence $(x_{1}, x_{2}, \ldots )$ in $G^{1}$ with 
$t(x_{n})= i(x_{n+1})$ for all $n \geq 1$. The bi-infinite path space $X_{G}$
is defined analogously with sequences indexed by the integers.
We let $\chi_{G}^{+}$ denote the obvious map from $X_{G}$ to $X_{G}^{+}$, which simply
restricts the domain of the sequence.

 The spaces $X^{+}_{G}$ and $X_{G}$ both have canonical  metrics. For 
 $x,y$ in  $X_{G}^{+}$, we define
\[
d_{G}(x,y) = \inf \{ 2^{-n} \mid n \geq 0, x_{i} = y_{i}, 1 \leq i \leq n \}.
\]
and for 
 $x,y$ in  $X_{G}$, we define
\[
d_{G}(x,y) = \inf \{ 2^{-n} \mid n \geq 0, x_{i} = y_{i}, 1-n \leq i \leq n \}.
\]
Using the same notation should cause no confusion. 
Observe that $\chi_{G}^{+}$ is a contraction.

The spaces $X_{G}^{+}$ and $X_{G}$ both carry dynamical systems. We define
$\sigma$ by $\sigma(x)_{n} = x_{n+1}$.
Here, either $x$ is in $X_{G}^{+}$ and $n \geq 1$ or $x$ is in $X_{G}$ and $n$ 
is an integer, making $\sigma$ a self-map of either space. Both are 
continuous and surjective and the latter is a  homeomorphism.
We make the easy observation that 
$d_{G}(\sigma^{n}(x), \sigma^{n}(y)) \leq 2^{n} d_{G}(x,y)$, 
for $x, y$ in $X_{G}^{+}$.

Let $G, H$ be two finite directed graphs. By a \emph{graph homomorphism}
 from $H$ 
to $G$ we mean a
 function, $\xi: H^{0} \cup H^{1} \rightarrow G^{0} \cup G^{1}$ ,
such that $\xi(H^{0}) \subseteq G^{0}, \xi(H^{1}) \subseteq G^{1}$ and satisfying 
$t_{G} \circ \xi|_{H^{0}} = \xi \circ t_{H}|_{H^{0}},
 i_{G} \circ \xi|_{H^{0}} = \xi \circ i_{H}|_{H^{0}}$. It 
 will usually not cause any confusion to drop the subscripts on $i, t$.
 We usually write such a  function as $\xi: H \rightarrow G$. A 
 \emph{graph embedding} is a graph homomorphism which is injective.

If $(X, d)$ is a metric space, $x$ is in $X$ and $\epsilon$ is positive, then 
we let $X(x,\epsilon)$ denote the ball centred at $x$ of radius 
$\epsilon$. If $A$ is any subset of $X$, we let $Cl(A)$ denote its closure.
 If $(X_{1}, d_{1})$ and $(X_{2}, d_{2})$
are metric spaces, we let 
\[
d_{1} \times d_{2}((x_{1}, x_{2}), (y_{1}, y_{2})) = 
d_{1}(x_{1}, y_{1})  +d_{2}(x_{2}, y_{2}),
\]
for all $x_{1}, y_{1}$ in $X_{1}$ and $x_{2}, y_{2}$ in $X_{2}$, 
which is a metric for the product space $X_{1} \times X_{2}$. 

\section{Construction}
\label{constr}

 In this section, we describe our basic construction.
 
 \begin{defn}
 \label{constr:10}
 For finite directed graphs, $G, H$,
 a pair of graph embeddings $\xi^{0}, \xi^{1}: H \rightarrow G$
 are said to satisfy
 \begin{enumerate}
\item (H0) if $\xi^{0}|_{H^{0}} =  \xi^{1}|_{H^{0}} $, 
\item (H1) if  $\xi^{0}(H^{1}) \cap  \xi^{1}(H^{1}) $ is empty, 
\item (H2) if, for every $y$ in $H^{1}$, there is 
$x$ in $G^{1}$ with 
$t(x) = t(\xi^{0}(y)), i(x) = i(\xi^{0}(y))$ and 
$x \notin \xi^{0}(H^{1}) \cup  \xi^{1}(H^{1})$.
\end{enumerate}
We say that $G, H, \xi= (\xi^{0}, \xi^{1})$
 satisfies the \emph{standing hypotheses} if $G$ is primitive
and $\xi$ satisfies (H0), (H1) and (H2). 
\end{defn}

The first two conditions will be essential for our construction. The third is 
a convenience and is not always needed. It also seems likely that the hypothesis
that $G$ is primitive can be weakened.

We observe a simple consequence of (H0) and (H1): if $y$ is any edge
 in $H^{1}$, then 
$\xi^{0}(y)$ and $\xi^{1}(y)$ are distinct, but have the same initial and 
terminal vertices in $G^{0}$.

\begin{defn}
\label{constr:20}
Let $ \xi^{0}, \xi^{1}: H \rightarrow G$ satisfy (H0) and (H1)
of  \ref{constr:10}.
\begin{enumerate}
\item 
We let $\xi(H^{1}) = \xi^{0}(H^{1}) \cup \xi^{1}(H^{1})$.
\item 
For $x$ in $\xi(H^{1})$, we let $\varepsilon(x) = 0, 1$ be 
such that 
$x $ is in $\xi^{\varepsilon(x)}(H^{1})$.
\item
We let $G_{\xi}$ be the graph with $G^{0}_{\xi} = G^{0}$ and $G^{1}_{\xi}$ 
obtained by identifying $\xi^{0}(y)$ with $\xi^{1}(y)$, for every $y$ 
in $H^{1}$. We also let $\tau_{\xi}$ denote the obvious quotient map from
$G^{1}$ to $G^{1}_{\xi}$ as well as the associated map from 
$X_{G}^{+}$ to $X_{G_{\xi}}^{+}$ and the map from 
$X_{G}$ to $X_{G_{\xi}}$.
\end{enumerate}
\end{defn}

\begin{defn}
\label{constr:30}
Let $ \xi^{0}, \xi^{1}: H \rightarrow G$ satisfy (H0) and (H1)
of  \ref{constr:10}.
\begin{enumerate}
\item For $x$ in $X_{G}^{+}$, let $\kappa(x)$ be the number of positive integers $n$ 
such that $x_{n}$ is \emph{not} in $\xi(H^{1})$, allowing the values $0$ and $\infty$.
\item For $ 0 \leq k \leq \infty$, let $X_{k}^{+} = \kappa^{-1}\{ k \}$.
\item For $x$ in $X_{G}^{+}$ with $\kappa(x) > 0$, we define $n(x)$ to be the least 
positive integer with $x_{n(x)}$  \emph{not} in $\xi(H^{1})$.
\end{enumerate}
\end{defn} 
 
 The following result is worth noting but its proof is trivial and we omit it.
 
 \begin{prop}
 \label{constr:40}
 \begin{enumerate}
 \item 
 The sets $X^{+}_{k}, 0 \leq k \leq \infty,$ are pairwise disjoint and, for 
 any $0 \leq k < \infty$,  $\cup_{j=0}^{k} X^{+}_{j}$ is closed in $X^{+}_{G}$.
  \item If $x, y$ in $X_{G}^{+}$ satisfy $\tau_{\xi}(x) = \tau_{\xi}(y)$, then 
  $\kappa(x) = \kappa(y)$ and  $ n(x) = n(y)$ if $\kappa(x) > 0$.
 \item 
 If $x$ is in $X_{G}^{+}$ with $0  < \kappa(x) < \infty $, 
 then $\kappa(\sigma^{n(x)}(x) ) = \kappa(x) -1$.
  \end{enumerate}
 \end{prop}
 
The following is a summary of some convenient topological 
properties of the sets $X_{k}^{+}, k \geq 0$.  Essentially, these are
consequences of our property (H2) and $G$ being primitive. 
 
 \begin{prop}
 \label{constr:50}
 Suppose that $G, H, \xi$ satisfy the standing hypotheses.
 \begin{enumerate}
  \item 
  If $x$ is in $X_{j}^{+}$  and $j < k < \infty$, then there is a sequence
   $x^{l}, l \geq 1,$ in $X^{+}_{k}$ which converges to $x$ in $X_{G}^{+}$.
   \item The closure of $X_{k}^{+}$ in $X_{G}^{+}$ is $\cup_{j=0}^{k} X_{j}^{+}$.
 \item
If
$x$ is in  $X_{G}^{+}$, then 
there is a sequence $x^{l}, l \geq 1,$ converging to $x$ in $X_{G}^{+}$ 
with 
$x^{l}$ in $X^{+}_{l}$, for all $l \geq 1$.
 \item The closure of 
 $\cup_{k=0}^{\infty} X_{k}^{+}$ in $X_{G}^{+}$ is $X_{G}^{+}$.
\end{enumerate} 
 \end{prop}
 
 \begin{proof}
 We first observe that 
 it follows from hypothesis (H2) and $G$ being primitive
  that $G^{1}-\xi(H^{1})$ is 
 connected. 
 
For $x$  in $X_{j}^{+}$, there are only finitely many $n$ with $x_{n}$ not 
in $\xi(H^{1})$. Choose $n_{0}$ such that $x_{n}$ is in $\xi(H^{1})$, for 
all $n \geq n_{0}$. For $l \geq 1$, we define $x^{l}$ in three segments, 
as follows.  Define $x^{l}_{ n} = x_{n}$, for 
$1 \leq n \leq n_{0}+l$.
Then choose $x^{l}_{n}, n_{0}+l < n \leq n_{0} + l + k-j $ to be any path in 
$G^{1}-\xi(H^{1})$ going through
 the same vertices as $x_{n}$, by property (H2). Finally, 
 set $x^{l}_{n} = x_{n}$ for $n > n_{0} + l + k-j$. It is clear 
 that this sequence has the desired properties. The first part 
 shows that the closure of $X_{k}^{+}$ contains $\cup_{j=0}^{k} X_{j}^{+}$. 
 The reverse containment follows from the fact that the 
 latter is closed, as we noted in \ref{constr:40}.

As $G$ is primitive,   we may find $l_{0} \geq 1$ such that, for any ordered pair
of vertices in $G^{0}$ and any $l \geq l_{0}$, 
there is a path of length $l$ between these two vertices. 
From property (H2), we may assume such a path lies in $G^{1}- \xi(H^{1})$.
For $l > l_{0}$, we again 
define $x^{l}$ in three segments as follows. For 
the first segment, we set
 $x^{l}_{n} = x_{n}$, for $1 \leq n \leq l-l_{0}$. Let $j$ be the number 
 of integers, $n$,
 between $1$ and $l-l_{0}$ with $x_{n}$ not in $\xi(H^{1})$. Obviously, 
 $j \leq l-l_{0}$. For 
 the second segment, $x^{l}_{n}, l-l_{0} < n \leq 2l - l_{0} - j $,
  we choose any path in $G^{1} - \xi(H^{1})$ from 
 $t(x_{l-l_{0}})$ to any vertex of $\xi^{0}(H^{0})$ having length
 $l - j \geq l_{0}$. Finally, we define the third segment 
 $x^{l}_{n}, 2 l-l_{0}-j < n < \infty, $ to be any path in $\xi(H^{1})$.
 The points $x^{l}, 1 \leq l \leq l_{0},$ may be chosen arbitrarily from 
 $X_{l}^{+}$. 
  Again, it is clear 
 that this sequence has the desired properties. The last statement follows at once.
 \end{proof}
 
We now give our two main definitions. 
The influence of binary expansion should be clear in the first.

\begin{defn}
\label{constr:60}
Let $x$  be in $X^{+}_{G}$.
\begin{enumerate}
\item If 
\[
x = (x_{1}, x_{2}, \cdots ,x_{n-1},
 \xi^{1-i}(y_{n}), \xi^{i}(y_{n+1}), \xi^{i}(y_{n+2}), \ldots),
 \]
 for some $n \geq 1$,  $y_{n}, y_{n+1}, \ldots$ in $H^{1}$ and $i=0,1$,
 then we define
 \[
x' = (x_{1}, x_{2}, \cdots, x_{n-1},
 \xi^{i}(y_{n}), \xi^{1-i}(y_{n+1}), \xi^{1-i}(y_{n+2}), \ldots).
 \]
 \item If 
\[
x = (x_{1}, x_{2},  \cdots ,
 x_{n}, \xi^{i}(y_{n+1}), \xi^{i}(y_{n+2}), \ldots),
 \]
 for some $n \geq 1$ and $i = 0,1$, $y_{n+1}, y_{n+2}, \ldots$ in $H^{1}$, 
 $x_{n}$ is not in $\xi(H^{1})$, and $i=0,1$,
 then we define
 \[
x' = (x_{1}, x_{2}, \cdots, x_{n},
 \xi^{1-i}(y_{n+1}), \xi^{1-i}(y_{n+2}), \xi^{1-i}(y_{n+3}), \ldots).
 \]
 \item We define $x \sim_{\xi} x'$, whenever $x, x'$ are as above. Also, we
 define
 $x \sim_{\xi} x$, for every $x$ in $X_{G}^{+}$.
 \end{enumerate}
 \end{defn}
 
 It is worth noting that the $n$ in parts 1 and 2 is unique (if it exists). 
 The proofs of the following easy facts
 are left to the interested reader.

 \begin{prop}
 \label{constr:70}
 \begin{enumerate}
 \item 
 $\sim_{\xi}$ is an equivalence relation and each equivalence class
 has either one or two elements.
 \item If $x \sim_{\xi} x'$, then 
 $\tau_{\xi}(x) = \tau_{\xi}(x')$.
 \item For each $k \geq 0$, the set $X_{k}^{+}$ is $\sim_{\xi}$-invariant.
 That is, if $x$ is in  $X_{k}^{+}$ and $x \sim_{\xi} x'$, then 
 $x'$ is in  $X_{k}^{+}$.
 \item If $x \sim_{\xi} x'$, then $\sigma(x) \sim_{\xi} \sigma(x')$.
 \item Applying $\sigma$  to the $\sim_{\xi}$-equivalence class of
 $x$ yields the $\sim_{\xi}$-equivalence class of $\sigma(x)$.
 
\end{enumerate}
 \end{prop}
 
 We can now define our main item of interest.
 
 \begin{defn}
 \label{constr:80}
 Let $G, H$ be finite directed graphs and  $\xi^{0}, \xi^{1}: H \rightarrow G$
 be two graph embeddings satisfying (H0) and (H1) of Definition \ref{constr:10}.
  We define 
 $X^{+}_{\xi}$ to be the quotient space of $X^{+}_{G}$ by the equivalence relation
 $\sim_{\xi}$, endowed with the quotient topology. We let 
 $\pi_{\xi}: X^{+}_{G} \rightarrow X^{+}_{\xi}$ be the quotient map.
 It follows from part 2 of Proposition \ref{constr:70} that there is a function
 we denote $\rho_{\xi}: X_{\xi}^{+} \rightarrow X_{G_{\xi}}$ satisfying
 $\rho_{\xi} \circ \pi_{\xi} = \tau_{\xi}$.
 
 We also define $\sigma_{\xi}: X^{+}_{\xi} \rightarrow X^{+}_{\xi}$ by 
$ \sigma_{\xi}(\pi_{\xi}(x) ) = \pi_{\xi}(\sigma(x))$, 
  for all $x$ in $X^{+}_{G}$, which is well-defined by 
  Proposition \ref{constr:70}.
   We let $X^{+}_{ \xi, k} = \pi_{\xi}(X^{+}_{k})$, for each $k \geq 0$. 
 \end{defn}

As indicated above, the space $X_{\xi}^{+}$ is given the quotient topology.
It is not hard to use general topological results to show that 
$X_{\xi}^{+}$ is metrizable. 
That is not sufficient for our goals as we will need a metric with 
specific properties. The 
construction is rather subtle.

 We begin the construction   of a 
pseudo-metric, $d$,  on $X^{+}_{G}$. We will show that
$d(x,y) \leq  3 d_{G}(x,y)$, for all $x, y$ in $X^{+}_{G}$.
 In addition, we will prove $d(x,y) = 0$ if and only if 
 $x \sim_{\xi} y$, for any $x,y$ in $X^{+}_{G}$. 
 The immediate consequence is that $d$ induces a metric
 on the quotient space $X^{+}_{\xi}$,  denoted 
 by $d_{\xi}$, whose topology is the the 
 quotient topology. This means that an alternate 
 definition of our space $X_{\xi}^{+}$ is as the metric space 
 naturally induced by the pseudo-metric $d$ on $X_{G}^{+}$. 
 The definition given in \ref{constr:60}  and \ref{constr:80}
 seems more intuitive.

 For $x$ in $X^{+}_{G}$, if $\kappa(x)=0$, we define
 \[
 \theta(x) = \exp \left(2 \pi i \sum_{j=1}^{\infty}
  \varepsilon(x_{j}) 2^{-j} \right).
 \] 
 If $\kappa(x) > 0$, we define 
 \[
 \theta(x) = \exp \left(2 \pi i \sum_{j=1}^{n(x)-1}
  \varepsilon(x_{j}) 2^{-j} \right).
  \]
  Obviously, $\theta(x)$  lies on the unit circle, $\T$, in either case and is 
  a $2^{n(x)-1}$-th root of unity in the latter.
 
 For $k \geq 0$, we 
 define $\lambda_{k}: X^{+}_{k} \times X^{+}_{k} \rightarrow [0, \infty)$ 
 inductively
 as follows.
 First, we set 
 \begin{eqnarray*}
 \lambda_{0}(x,y) & = &  d_{\T}(\theta(x), \theta(y)) \\
   &  =  &  \inf\{ \vert t \vert \mid t \in \R, 
    e^{2\pi i t} \theta(x) = \theta(y) \},
    \end{eqnarray*}
     for 
 $x,y$ in $X^{+}_{0}$.

 For $x,y$ in $X^{+}_{k}$ with $k > 0$, we define 
 \begin{eqnarray*}
 \lambda_{k}(x,y) & =  & \vert 2^{-n(x)} - 2^{-n(y)} \vert 
  + d_{\T}(\theta(x), \theta(y) ) \\
  &  &  +  \delta((n(x),\theta(x)), (n(y), \theta(y))) 
 2^{-2-n(x)} \lambda_{k-1}(\sigma^{n(x)}(x),\sigma^{n(y)}(y)),
  \end{eqnarray*}
   where we use $\delta( \cdot, \cdot)$ for the Kronecker delta function.
   
  \begin{lemma}
  \label{constr:90}
  For $k \geq 0$, we have 
  \[
  d_{\T}(\theta(x), \theta(y)) \leq  d_{G}(x,y),
  \]
  and 
  \[
\lambda_{k}(x,y) \leq   2 d_{G}(x,y),
\]
for all $x,y$ in $X^{+}_{k}$.
  \end{lemma}
  
  \begin{proof}
  We proceed by induction on $k$. 
  For $k =0$, suppose $d_{G}(x,y) = 2^{-m}$, for some $m \geq 0$,
  which implies $x_{j} = y_{j}$ for $1 \leq j \leq m$. We have 
  \begin{eqnarray*}
  d_{\T}(\theta(x), \theta(y)) & \leq & \sum_{j=1}^{\infty} 
  \vert \varepsilon(x_{j}) - \varepsilon(y_{j}) \vert 2^{-j}  \\
    &  \leq &  \sum_{j=m+1}^{\infty} 
  \vert \varepsilon(x_{j}) - \varepsilon(y_{j}) \vert 2^{-j} \\
    &  \leq & 2^{-m} \\
    &  = &  d_{G}(x,y).
    \end{eqnarray*}
  For $k=0$, the second statement follows from the first as 
  \newline
   $\lambda_{0}(x,y) = d_{\T}(\theta(x), \theta(y))$.

  Now assume that $k >0$ and again $d_{G}(x, y) = 2^{-m}$, so 
 $x_{j} = y_{j}$ for $1 \leq j \leq  m$. If
  either $n(x) \leq  m $ or $n(y) \leq m$, 
  then $n(x) = n(y)$
 and $\vert 2^{-n(x)} - 2^{-n(y)} \vert=0$. In addition, we have
 $\theta(x) = \theta(y)$ so $d_{\T}(\theta(x), \theta(y))=0$. 
 
  Otherwise, we have $n(x), n(y) > m$ 
  and 
  \begin{eqnarray*}
  d_{\T}(\theta(x), \theta(y)) & \leq  &  \vert \sum_{j=m+1}^{n(x)} 
  \varepsilon(x_{j}) 2^{-j}
 -  \sum_{j=m+1}^{n(y)} \varepsilon(y_{j})2^{-j} \vert \\
 &  \leq   & \sum_{j=m+1}^{\infty} 2^{-j} \\
  & =  & 2^{-m} \\
  & = &  d_{G}(x,y).
  \end{eqnarray*}  
 This establishes the first inequality. In addition, we note that
 \[
 \vert 2^{-n(x)} - 2^{-n(y)} \vert \leq 2^{-m-1} <  d_{G}(x,y).
 \]

 Now, we consider $\lambda_{k}(x,y)$. 
If either $n(x) \neq n(y)$ or $\theta(x) \neq \theta(y)$, then we have 
\[
\lambda_{k}(x,y) =  \vert 2^{-n(x)} - 2^{-n(y)} \vert 
  + d_{\T}(\theta(x), \theta(y) ) \leq d_{G}(x,y) +  d_{G}(x,y).
  \]

  In the case   $n(x) = n(y)=n$ and $\theta(x) = \theta(y)$,
  from the definition and the induction hypothesis, 
  we have
 \begin{eqnarray*}
  \lambda_{k}(x,y) & = & 2^{-2-n} \lambda_{k-1}(\sigma^{n}(x), \sigma^{n}(y)) \\
  &  \leq  &  2^{-2-n} \cdot 2 d_{G}(\sigma^{n}(x), \sigma^{n}(y))  \\
   & \leq & 2^{-2-n} \cdot 2 \cdot 2^{n}   d_{G}(x, y)\\
  & < &  d_{G}(x, y).
  \end{eqnarray*}
  \end{proof}
  
  We can now begin our definition of our pseudo-metric on $X_{G}^{+}$. Initially, 
  we treat the spaces $X_{k}^{+}$ separately.

  \begin{prop}
    \label{constr:100}
  \begin{enumerate}
  \item 
  For each $k \geq 0$, 
  \[ 
  d_{k}(x,y) = d_{G_{\xi}}(\tau_{\xi}(x), \tau_{\xi}(y)) + \lambda_{k}(x,y),
  \]
  for $x, y$ in $X^{+}_{k}$, is a pseudo-metric on $X^{+}_{k}$.
   \item For $x, y$ in $X^{+}_{k}$, we have 
   \[
    d_{G_{\xi}}(\tau_{\xi}(x), \tau_{\xi}(y)) \leq 
    d_{k}(x,y) \leq 3 d_{G}(x,y).
    \]
  \item For $x, y$ in $X^{+}_{k}$, $d_{k}(x,y)=0$ if and only if 
  $x \sim_{\xi} y $.
  \end{enumerate}
  \end{prop}

  \begin{proof}
  For the first part, the $d_{G_{\xi}}$ term is clearly a pseudometric, 
  so it suffices to show that $\lambda_{k}$ is also, which 
  we do 
  by induction on $k$. This is clear
  for $k=0$. For $k > 0$, a  simple induction argument
  shows that $\lambda_{k}$  is reflexive and symmetric. Now suppose that
  $x, y, z$ are in $X^{+}_{k}$. We first consider the 
  case  $(n(x), \theta(x) ) \neq (n(y), \theta(y))$,
  where we have 
  \begin{eqnarray*}
  \lambda_{k}(x,y)  & =  & \vert 2^{-n(x)} - 2^{-n(y)} \vert 
    +  d_{\T}(\theta(x), \theta(y) ) \\
    &  \leq  & \vert 2^{-n(x)} - 2^{-n(z)} \vert 
    +  d_{\T}(\theta(x), \theta(z))  \\
    &  & \vert 2^{-n(z)} - 2^{-n(y)} \vert 
    +  d_{\T}(\theta(z), \theta(y) ).
    \end{eqnarray*}
If the first two terms are both zero, then their sum is certainly less 
than or equal to $\lambda_{k}(x,z)$. On the other hand, if either is 
non-zero, then their  sum equals  $\lambda_{k}(x,z)$, by definition.
    The third and fourth terms are done similarly.

  Now suppose that  $(n(x), \theta(x) ) = (n(y), \theta(y))$ . If \newline
  $(n(z), \theta(z) ) = (n(x), \theta(x))$
   then, 
  \begin{eqnarray*}
  \lambda_{k}(x,y) & = &   2^{-2-n(x)}
  \lambda_{k-1}(\sigma^{n(x)}(x), \sigma^{n(y)}(y) ) \\
     &  \leq  &    2^{-2-n(x)} \lambda_{k-1}(\sigma^{n(x)}(x), \sigma^{n(z)}(z) ) \\
       &  &  +
    2^{-2-n(x)}   \lambda_{k-1}(\sigma^{n(z)}(z), \sigma^{n(y)}(y) ) \\
      &  =  &  \lambda_{k}(x,z) + \lambda_{k}(z,y)
     \end{eqnarray*}
     follows from the induction hypothesis.

   We next consider the case $n(z) \neq n(x)$ and so $n(z) \neq n(y)$ also.
 For any positive 
     integers $m\neq n$, we have 
     \[
     \vert 2^{-m} - 2^{-n} \vert = 2^{-\min\{ m,n \}} 
     \vert 1 - 2^{-\vert m-n \vert} \vert \geq 2^{-\min\{ m,n \}-1} \geq 2^{-n-1}
     \]
     and so
     \[ 
     \lambda_{k}(x,z) \geq \vert 2^{-n(x)} - 2^{n(z)} \vert \geq 2^{-n(x)+1}.
     \] 
   A similar estimate holds for $\lambda_{k}(z,y)$.
   
     On the other hand, we also have 
   \begin{eqnarray*}
  2^{-2-n(x)} \lambda_{k-1}(\sigma^{n(x)}(x), \sigma^{n(y)}(y) ) 
    &  \leq  &  2^{-2-n(x)} 2d_{G}( \sigma^{n(x)}(x), \sigma^{n(y)}(y) ) \\
     &  \leq  & 2^{-1-n(x)}.
       \end{eqnarray*}
       
   Together this yields
   \begin{eqnarray*}
   \lambda_{k}(x,y) & = &  
   2^{-2-n(x)} \lambda_{k-1}(\sigma^{n(x)}(x), \sigma^{n(y)}(y) ) \\
     &  \leq &  2^{-1-n(x)} \\
       &  =  & 2^{-2-n(x)} +  2^{-2-n(y)} \\
       & \leq & \lambda_{k}(x,z) + \lambda_{k}(z,y).
   \end{eqnarray*}
       
       The final case to consider is $n(x) = n(y) = n(z)$
       and  $\theta(z) \neq \theta(x)$, so $\theta(z) \neq \theta(y)$
       also. 
      Then, as $\theta(x), \theta(y), \theta(z)$ are all 
      $2^{n(x)-1}$-th roots of unity, we have 
       \[
       d_{\T}( \theta(x), \theta(z)), d_{\T}( \theta(z), \theta(y))
          \geq 2^{-n(x)}.
         \]
          It follows that
    \begin{eqnarray*}
   \lambda_{k}(x,y) & = &  
   2^{-2-n(x)} \lambda_{k-1}(\sigma^{n(x)}(x), \sigma^{n(y)}(y) ) \\
     &  \leq &  2^{-1-n(x)} \\
       &  =  & 2^{-2-n(x)} +  2^{-2-n(y)} \\
       & \leq & d_{\T}( \theta(x), \theta(z)) + d_{\T}( \theta(z), \theta(y)) \\
       & = & \lambda_{k}(x,z) + \lambda_{k}(z,y).
   \end{eqnarray*} 
     
     The second part is immediate from 
      the second statement of Lemma \ref{constr:90}.
      
      For the third part, we consider the 'if' direction first.
We proceed by induction on $k$, beginning with $k=0$. 
    Suppose $ x \sim_{\xi} y$,
       where $x,y$ are in $X^{+}_{0}$. From part  2 of 
       Proposition \ref{constr:40}, we know
        $\tau_{\xi}(x) = \tau_{\xi}(y)$.
   In addition, $\varepsilon(x_{j}), j \geq 1$ and   
      $\varepsilon(y_{j}), j \geq 1$ are two $0,1$-sequences and 
    the definition of $\sim_{\xi}$  implies
      that these sequences are dyadic representations of the 
      same real number, modulo the integers. This implies 
      $\theta(x) = \theta(y)$ and so 
      \[
      d_{0}(x,y) = d_{\T}(\theta(x), \theta(y))
       + d_{G_{\xi}}(\tau_{\xi}(x), \tau_{\xi}(y)) = 0 +0=0.
      \]
      
      Now assume the result is true for some $k$ and let $x \sim_{\xi} y$ 
      be in $X^{+}_{k+1}$. From the Proposition \ref{constr:40},
       we have $n(x) = n(y)$ and 
      $x_{i} = y_{i}$, for all $ 1 \leq i \leq n(x)$.  It follows that
      $\theta(x) = \theta(y)$ and 
      \[
      d_{k+1}(x,y) = 2^{-2-n(x)} \lambda_{k}(\sigma^{k}(x), \sigma^{k}(y))
       \leq 2^{-2-n(x)} d_{k}( \sigma^{k}(x), \sigma^{k}(y)).
       \]
      It follows from Proposition \ref{constr:40}
       $\sigma^{k}(x) \sim_{\xi} \sigma^{k}(y)$,
      so the conclusion follows from the induction hypothesis.
      
      Now we consider the 'only if' direction.
        Suppose $d_{0}(x,y)=0 $,  where $x,y$ are in $X^{+}_{0}$.
          This means that 
      $x_{i}, y_{i}$ are in $\xi(H^{1})$, for all $i \geq 1$. 
      The hypothesis obviously implies that  
      $d_{G_{\xi}}(\tau_{\xi}(x), \tau_{\xi}(y)) =0$
      so $\tau_{\xi}(x_{i}) = \tau_{\xi}(y_{i})$, for all $i \geq 1$.
      In other words, there is  a path $z$ in $X^{+}_{H}$ such that 
      $\{ x_{i}, y_{i} \} \subseteq \{ \xi^{0}(z_{i}), \xi^{1}(z_{i}) \}$
      for all $i$. The fact that $\lambda_{0}(x,y)=0$ implies
      that $\theta(x) = \theta(y)$ and this implies that the sequences
      $\varepsilon(x_{i}), \varepsilon(y_{i}), i \geq 1$ are 
      the binary expansions of the same real number, modulo the integers.
      This implies $x \sim_{\xi}   y$.

      We now consider $k > 0$ and  $d_{k}(x,y)=0$, where $x,y$ are in $X^{+}_{k}$.
      The first consequence is that $\tau_{\xi}(x) = \tau_{\xi}(y)$. 
      This implies that $n(x) = n(y) = n$ and since $x_{n(x)}$ and $y_{n(y)}$ are 
      not in $\xi^{0}(H) \cup \xi^{1}(H)$, we have $x_{n(x)} = y_{n(y)}$. 
      In addition, there is a path $h_{1}, \ldots, h_{n-1}$ such that 
      $x_{i}, y_{i}$ are in $\{ \xi^{0}(h_{i}), \xi^{1}(h_{i}) \}$,
       for all $1 \leq i < n$. The fact that $\lambda_{k}(x,y)=0$ implies 
       that $\theta(x) = \theta(y)$. The map which takes a $0,1$  sequence
       $\varepsilon_{1}, \ldots, \varepsilon_{n-1}$ to 
       $exp (2 \pi \sum_{j=1}^{n-1} \varepsilon_{j}2^{-j})$ is injective
       and so we conclude that $\varepsilon(x_{i}) = \varepsilon(y_{i})$, for 
       all $1 \leq i < n$. It follows that $x_{i} = y_{i}$, for $1 \leq i \leq n$.
       The fact $\tau_{\xi}(x) = \tau_{\xi}(y)$ implies that 
       $\tau_{\xi}(\sigma^{n(x)}(x)) = \tau_{\xi}(\sigma^{n(y)}(y))$ also and
       so
    we have 
       \[
       d_{k-1}(\sigma^{n}(x),\sigma^{n}(y)) = 
       \lambda_{k-1}(\sigma^{n}(x),\sigma^{n}(y))
       = 2^{2+n} \lambda_{k}(x,y) =0.
       \]
       By our induction hypothesis, 
       $\sigma^{n}(x) \sim_{\xi} \sigma^{n}(y) $. Together with
       the fact that
       $x_{i} = y_{i}$ for $ 1 \leq i \leq n$ implies 
       $x \sim_{\xi} y$.
   \end{proof}

      It will be useful to have a characterization of 
       Cauchy sequences in our pseudo-metric.

   \begin{prop}
    \label{constr:120}
  \begin{enumerate}
  \item A sequence $x^{l}, l \geq 1,$ in $X^{+}_{0}$ is Cauchy 
  in $d_{0}$ if and only if 
  $\tau_{\xi}(x^{l}), l \geq 1,$ is convergent in 
  $X^{+}_{G_{\xi}}$ and  $\theta(x^{l}), l \geq 1$ is convergent in $\C$.
  \item 
   For 
  $k > 0$, a sequence $x^{l}, l \geq 1,$ in $X^{+}_{k}$ is Cauchy 
  in $d_{k}$ if and only only $\tau_{\xi}(x^{l}), l \geq 1,$ 
  is convergent in $X^{+}_{G_{\xi}}$ and either 
  \begin{enumerate}
  \item $\lim_{l} n(x^{l}) = \infty$ and  $\theta(x^{l})$, 
  $ l \geq 1,$ is convergent in $\C$
   or
  \item the sequences $n(x^{l})$ and $\theta(x^{l})$ are eventually constant and 
   \newline 
   $\sigma^{n(x^{l})}(x^{l}), l \geq 1,$ is Cauchy in $d_{k-1}$.
  \end{enumerate}
  \end{enumerate}
  \end{prop}

    \begin{proof}   
   The first part is clear from the definition of $d_{0}$
   
     For the second part we begin with the 'only if' direction. 
     If $x^{l}, l \geq 1$, is Cauchy in $d_{k}$, then $\tau_{\xi}(x^{l}), l \geq 1$
      is clearly 
     Cauchy and hence convergent in $X_{G_{\xi}}$. In addition,
       both sequences $ 2^{-n(x^{l})}$ and $\theta(x^{l})$ are Cauchy.
    There are clearly two cases:
      either $n(x^{l}), l \geq 1$, tends to infinity or it is eventually 
      constant.
      In the first case, we have case (a).
       In the second case, if $n(x^{l})=n$ for all 
      $l$ sufficiently large, then for such $l$, $\theta(x^{l})$ lies 
      in a finite set, so to be Cauchy it must be eventually constant also, say 
      with value $\theta$. Then for sufficiently large values   of $l, m$, we have 
      \begin{eqnarray*}
      d_{k}(x^{l}, x^{m}) & = & d_{G_{\xi}}(\tau_{\xi}(x^{l}), \tau_{\xi}(x^{m}))
      +  2^{-2-n} \lambda_{k-1}(\sigma^{n}(x^{l}), \sigma^{n}(x^{m})) \\
        &  \geq & 2^{-n} d_{G_{\xi}}(\tau_{\xi}(\sigma^{n}(x^{l})), 
        \tau_{\xi}(\sigma^{n}(x^{m}))) \\
       & &   +  2^{-2-n} \lambda_{k-1}(\sigma^{n}(x^{l}), \sigma^{n}(x^{m})) \\
          & \geq & 2^{-2-n} d_{k-1}(\sigma^{n}(x^{l}), \sigma^{n}(x^{m})).
      \end{eqnarray*}
      It follows that $\sigma^{n}(x^{l}), l \geq 1$ is Cauchy in $d_{k-1}$.
      
      Conversely, if $x^{l}, l \geq 1$ is any sequence in $X^{+}_{k}$,
       then,
      for any $l,m$, $d_{k}(x^{l}, x^{m})$ is either equal to either
      \[
      d_{G_{\xi}}( \tau_{\xi}(x^{l}), \tau_{\xi}(x^{m}))
        +  \vert 2^{-n(x^{l})} - 2^{-n(x^{m})} \vert 
        + d_{\T}(\theta(x^{l}), \theta(x^{m})),
        \]
        or 
        \[
  d_{G_{\xi}}(\tau_{\xi}(x^{l}), \tau_{\xi}(x^{m}))
      + 2^{-2-n(x^{l})}
       \lambda_{k-1}(\sigma^{n(x^{l})}(x^{l}), \sigma^{n(x^{l})}(x^{m})).      
      \] 
      It is then immediate that if $\tau_{\xi}(x^{l}), l \geq 1$ being 
      convergent and either of the
       two conditions imply the sequence is Cauchy in $d_{k}$.
  \end{proof}
  
 At this point, we have defined a pseudo-metric on  $X_{k}^{+}$, for each 
 value of $k$. Assuming condition (H2) and that $G$ is primitive, 
 we also know that $X_{k}^{+}$ will have limit points
 in $X_{j}^{+}$, for every $ 0 \leq j < k$. We now need to show
 these different pseudo-metrics are compatible, in an obvious sense. The following 
 concerns a special case ($j=0$) but will be used in the proof
 of the general statement, as well.

     \begin{lemma}
  \label{constr:130} 
  Let $k \geq 1$ and suppose that $x^{l}, l \geq 1,$ is a 
  sequence in $X^{+}_{k}$
   which converges to $x$ in $X_{G}^{+}$ and
  such that $n(x^{l})$ tends to infinity as $l$ tends to infinity. Then 
  $x$ is in $X^{+}_{0}$ and 
  $\lim_{l \rightarrow \infty} \theta(x^{l}) = \theta(x)$.
  \end{lemma}

  \begin{proof}
For any $n \geq 1$, we have 
$x_{n} = \lim_{l} x_{n}^{l}$ and $x_{n}^{l}$ is in $\xi(H^{1})$, provided
$n(x^{l}) > n$. It follows that $x_{n}$ is in $\xi(H^{1})$, for every $n$, 
so $x$ is in $X^{+}_{0}$.

If we fix $m \geq 1$, we may choose $l_{0}$ sufficiently large so  
that $n(x^{l}) > m$ and $x_{i} = x_{i}^{l}$, for all 
$1 \leq i \leq m$ and $l \geq l_{0}$. For such $l$, we have
\begin{eqnarray*}
\vert \theta(x) - \theta(x^{l})   \vert
 & = & \vert \exp( 2 \pi i \sum_{j=1}^{\infty}  \varepsilon(x_{j}) 2^{-j}) - 
 \exp( 2 \pi i \sum_{j=1}^{n(x^{l})-1}  \varepsilon(x^{l}_{j}) 2^{-j}) \vert \\
 & = & \vert \exp( 2 \pi i \sum_{j=m+1}^{\infty}  \varepsilon(x_{j}) 2^{-j}) - 
 \exp( 2 \pi i \sum_{j=m+1}^{n(x^{l})-1}  \varepsilon(x^{l}_{j}) 2^{-j}) \vert \\
  & \leq &2 \pi  \vert \sum_{j=m+1}^{\infty}  \varepsilon(x_{j}) 2^{-j} 
  - \sum_{j=m+1}^{n(x^{l})-1}  \varepsilon(x^{l}_{j}) 2^{-j} \vert \\
  & \leq & 2 \pi 2^{-m}.
  \end{eqnarray*}
  As $m$ was arbitrary, this completes the proof. 
  \end{proof}
 
 The following summarizes the compatibility of the different 
 \newline pseudo-metrics.
      
   \begin{prop} 
   \label{constr:140}
   Suppose $ 0 \leq j <  k$, the  sequences $x^{l}, l \geq 1$  and 
   $y^{l}, l \geq 1$ are in $X^{+}_{k}$, $x$ and $y$ are in
  $X^{+}_{j}$, $x^{l}, l \geq 1,$ converges to 
  $x$  and  $y^{l}, l \geq 1,$ converges to 
  $y$ in $X^{+}_{G}$, then 
  \[
  d_{j}(x,y) = \lim_{l \rightarrow \infty} d_{k}(x^{l}, y^{l}).
  \]  
  \end{prop} 
 
  \begin{proof}      
      We first consider the case $n(x^{l}), l \geq 1,$ tends
      to infinity.  The fact that $x^{l},l \geq 1$, converges in 
      $X^{+}_{G}$, along with part 2 of Proposition  \ref{constr:100}
      shows that $x^{l}, l \geq 1$, is Cauchy in 
      $d_{k}$. Proposition \ref{constr:120}
       shows that there are two cases to consider.  We begin 
      by assuming that (a) holds so $n(x^{l})$ tends to infinity while 
      $\theta(x^{l}), l \geq 1,$  converges.

       Lemma \ref{constr:130}
      implies that $x$ is in $X^{+}_{0}$.
       Hence, we have $j=0$. We now claim that
      the sequence $n(y^{l})$ must also tend to infinity. By the same 
      argument as above, $y^{l}, l \geq 1$ is Cauchy in $d_{k}$. 
      If $n(y^{l}), l \geq 1$ is eventually constant, say $n$, 
      then $y_{n}= \lim_{l} y_{n}^{l}$ is not in $\xi(H^{1})$ which means
       that $y$ is not
      in $X^{+}_{0}$, a contradiction.

      Now, we consider the case 
      $\theta(x) \neq \theta(y)$. From Lemma \ref{constr:130}, we see that
      $\theta(x^{l}) \neq \theta(y^{l})$ for $l$ sufficiently large. 
      In this case, we know that
      \[
      d_{k}(x^{l}, y^{l}) = d_{G_{\xi}}(\tau_{\xi}(x^{l}), \tau_{\xi}(y^{l})) 
      + \vert 2^{-n(x^{l})} - 2^{-n(y^{l})} \vert +
      d_{\T}( \theta(x^{l}), \theta(y^{l})).
      \]
      The first term converges to $d_{G_{\xi}}(\tau_{\xi}(x), \tau_{\xi}(y))$, 
      the second to zero and the third to $d_{\T}( \theta(x), \theta(y))$. 
      The desired equality follows from the definition of $d_{0}(x,y)$.

      We next turn to the case $\theta(x)=\theta(y)$. Here, we have either 
       \[
      d_{k}(x^{l}, y^{l}) = d_{G_{\xi}}(\tau_{\xi}(x^{l}), \tau_{\xi}(y^{l})) 
      + \vert 2^{-n(x^{l})} - 2^{-n(y^{l})} \vert +
      d_{\T}( \theta(x^{l}), \theta(y^{l})).
      \]
      or else
       \[
      d_{k}(x^{l}, y^{l}) = d_{G_{\xi}}(\tau_{\xi}(x^{l}), \tau_{\xi}(y^{l})) 
      + 2^{-2-n(x^{l})} 
      \lambda_{0}(\sigma^{n(x^{l})}(x^{l}), \sigma^{n(y^{l})}(y^{l})).
      \]
      In either case, the limit is 
      \[
      d_{G_{\xi}}(\tau_{\xi}(x), \tau_{\xi}(y)) =  
       d_{G_{\xi}}(\tau_{\xi}(x), \tau_{\xi}(y)) + d_{\T}(\theta(x), \theta(y))
       = d_{0}(x,y).
       \]
       
       We are left to consider the case that $n(x^{l})$ is
       eventually constant. It follows that $\theta(x^{l})$ is also. 
       Suppose these values are $n$ and $\theta$, respectively. It 
       follows that $n(x)=n, \theta(x) = \theta$ and $j >0$. 
      If the sequence $n(y^{l})$ is unbounded, then, arguing as before
      for $n(x^{l})$, we have $j=0$, which is a contradiction. Hence
      the sequences $n(y^{l})$ and $\theta(y^{l})$ are also eventually constant, 
      say with values $n'$ and $\theta'$, respectively.
      
      As $x^{l}_{i}$ ($y^{l}_{i}$) converges to 
      $x_{i}$ ($y_{i}$, respectively), for $1 \leq i \leq n$ ($1 \leq i \leq n'$,
      respectively, we know then that $n(x) = n$, $\theta(x) = \theta$, $n(y)=n'$
      and $\theta(y) = \theta'$.  Let us first assume that 
      $(n, \theta) \neq (n', \theta')$.
       In this case, we have, for $l$ sufficiently large, 
       \[
       d_{k}(x^{l}, y^{l})=d_{G_{\xi}}(\tau_{\xi}(x^{l}),\tau_{\xi}(y^{l}) )
         +  \vert 2^{-n} - 2^{-n'} \vert +
          d_{\T}(\theta, \theta') 
          \]
          which clearly converges to
    \[  
     d_{G_{\xi}}(\tau_{\xi}(x),\tau_{\xi}(y) )
         +  \vert 2^{-n} - 2^{-n'} \vert +
          d_{\T}(\theta, \theta')     =  d_{j}(x, y).
          \]

          The second case to consider is  $(n, \theta) = (n', \theta')$.
          In this case, we have, for $l$ sufficiently large,
     \begin{eqnarray*}
       d_{k}(x^{l}, y^{l}) & =  & 
       d_{G_{\xi}}(\tau_{\xi}(x^{l}),\tau_{\xi}(y^{l}) ) \\
  & &     +  2^{-2-n} \lambda_{k-1}(\sigma^{n}(x^{l}), \sigma^{n}(y^{l})) \\
         & = & d_{G_{\xi}}(\tau_{\xi}(x^{l}),\tau_{\xi}(y^{l}) ) \\
  &  &        -   2^{-2-n}
   d_{G_{\xi}}(\tau_{\xi}(\sigma^{n}(x^{l})),\tau_{\xi}(\sigma^{n}(y^{l})) ) \\
    &   &   +   2^{-2-n} d_{k-1}(\sigma^{n}(x^{l}), \sigma^{n}(y^{l})).
          \end{eqnarray*}
          It is clear that the 
           $\sigma^{n}(x^{l}), \sigma^{n}(y^{l}), l \geq 1$       
          and points $\sigma^{n}(x), \sigma^{n}(y)$ satisfy the 
          hypotheses for integer for $0 \leq j-1 < k-1$, so by the induction 
          hypothesis, the sequence above converges to 
       \begin{eqnarray*}
   &       d_{G_{\xi}}(\tau_{\xi}(x),\tau_{\xi}(y) )  & \\
      &    -   2^{-2-n}
         d_{G_{\xi}}(\tau_{\xi}(\sigma^{n}(x)),\tau_{\xi}(\sigma^{n}(y)) ) & \\
        &  +   2^{-2-n} d_{j-1}(\sigma^{n}(x), \sigma^{n}(y)) & \\
      =  & d_{j}(x,y). &
      \end{eqnarray*} 
  \end{proof}

    We are now ready to unify our metrics $d_{k}$ into a single pseudo-metric.

  Let us take a moment to recall (one approach to)
  taking the  completion of a metric space $(X, d)$ (see 
  page 196 of Kelley \cite{Ke:book}). 
  A key step is showing that, if $x^{l}, y^{l}, l \geq 1$ are 
  two Cauchy sequences in $X$ with respect to $d$, then the sequence 
  of real numbers, $d(x^{l}, y^{l}), l \geq 1,$ is also Cauchy and hence
  convergent. This proof only needs symmetry and the triangle inequality, and 
   works
  equally well if $d $ is a pseudo-metric.
  
  Let us describe how we will use this. For $x,y$ in
   $X_{k}^{+}$, we have the inequality $d_{k}(x,y) \leq 3 d_{G}(x,y)$. If we choose
   any sequences $x^{l}, y^{l}, l \geq 1,$ in $X_{k}^{+}$,
    converging to $x,y$ in $X_{G}^{+}$, 
   respectively, then both are Cauchy in $d_{G}$ and hence also in $d_{k}$.
   The argument outlined above then implies that 
   $\lim_{l \rightarrow \infty} d_{k}(x^{l},y^{l})$ exists and it 
   can be shown to be
    independent of the choice of sequences. If $x,y$ are actually in $X_{k}^{+}$, 
    then this limit agrees with $d_{k}(x,y)$. We refer to this as an 
    \emph{extension} of $d_{k}$. We use this repeatedly in the following.

  \begin{thm}
  \label{constr:160}
   Assume that  $G, H, \xi$ satisfies 
  the standing hypotheses. \newline
There exists a pseudo-metric, $d$, on $X_{G}^{+}$ satisfying the following.
\begin{enumerate}
\item 
\[
d_{G_{ \xi}}(\tau_{\xi}(x), \tau_{\xi}(y)) \leq 
d(x,y) \leq 3 d_{G}(x,y),
\]
 for all $x, y$ in $X_{G}^{+}$, 
\item $d(x,y) = 0$ if and only if $x \sim_{\xi} y$, for all $x, y$ in 
$X_{G}^{+}$,
\item $d(x,y) = d_{k}(x,y)$, for all $x, y$ in $X_{k}^{+}$.
\end{enumerate}
  \end{thm}

  \begin{proof}
 For $k \geq 0$, the space $X_{k}^{+}$ carries two pseudo-metrics, $d_{G}$ and 
 $d_{k}$. The former dominates the latter (part 2 of Proposition \ref{constr:100}), 
 at least up to a factor of $3$. In addition, it is clear that 
 $d_{G_{\xi}}(\tau_{\xi}(x), \tau_{\xi}(y)) \leq 
d_{k}(x,y)$, for all $x, y$ in $X_{k}^{+}$.
 We also recall from Proposition \ref{constr:50} that the closure 
 of $X_{k}^{+}$ in $X_{G}^{+}$ is $Cl(X_{k}^{+}) = \cup_{j=0}^{k} X_{j}^{+}$. 
 It follows that $d_{k}$ extends to a pseudo-metric on 
 $Cl(X_{k}^{+})$, which we denote $\bar{d}_{k}$, as described above
  which is also 
 dominated by $3d_{G}$ and bounded below by 
 $d_{G_{\xi}}(\pi_{\xi}(\cdot),\pi_{\xi}(\cdot))$. In addition, it follows from 
 Proposition \ref{constr:140} that the restriction of $\bar{d}_{k}$ to 
 $X_{j}^{+} \times X_{j}^{+}$ equals $d_{j}$, for any $0 \leq j \leq k$.
 It  also follows that 
  $\bar{d}_{k}$ equals $\bar{d}_{j}$  after extending the latter to 
  $\cup_{i=0}^{j} X_{i}^{+}$. 
 
 We may then define a pseudo-metric, $d_{\infty}$, on $\cup_{j=0}^{\infty}  X_{j}^{+}$
 which is $\bar{d}_{k}$  on $\cup_{j=0}^{k}  X_{j}^{+}$. 
 This is also bounded by $3d_{G}$ and so extends continuously to a
 pseudo-metric $d$ on 
 the closure of $\cup_{j=0}^{\infty}  X_{j}^{+}$ which is $X_{G}^{+}$ 
 (Proposition \ref{constr:40}) and is bounded below by 
 $d_{\xi}(\pi_{\xi}(\cdot),\pi_{\xi}(\cdot))$ and above by $3 d_{G}$.
 
 The first and third properties of   $d$ in the conclusion
  are immediate. 
 For the second, we first suppose that $x \sim_{\xi} y$. It follows 
 from the definition of $\sim_{\xi}$ that $x,y$ lie in $X_{k}^{+}$, for 
 some $0 \leq k < \infty$ and hence
$ d(x,y) = d_{k}(x,y) = 0$, by Proposition \ref{constr:100}.

Now assume that $x,y$ are in $X_{G}^{+}$ and $d(x,y) = 0$. 
It follows that $\tau_{\xi}(x) = \tau_{\xi}(y)$ which, in turn 
implies that $\kappa(x) = \kappa(y)$. If this is finite, say $k$, then 
$x,y$ are in $X_{k}^{+}$ and the conclusion follows from 
Proposition \ref{constr:100}. We now assume $\kappa(x) = \infty$ and 
  $\tau_{\xi}(x) = \tau_{\xi}(y)$ also implies that $n(x) = n(y)$.
  From Proposition \ref{constr:50}, we may find sequences,  $x^{l}, y^{l} $
  in $X_{l}^{+}$ converging to $x,y$, respectively. 
  This means that 
  \[
  0 = d(x,y) = \lim_{l} d(x^{l}, y^{l}) = \lim_{l} d_{l}(x^{l}, y^{l}).
  \]
  There exists $l_{0} \geq 1$
  such that, for all $l \geq l_{0}$, $x_{n} =x^{l}_{n}, y_{n} = y^{l}_{n}$, 
  for $1 \leq n \leq n(x)$. From this, it follows that 
  $n(x^{l}) = n(x) = n(y) = n(y^{l})$ and 
  $\theta(x) = \theta(x^{l}), \theta(y) = \theta(y^{l})$, for $l \geq l_{0}$.
  If $\theta(x^{l})$ and $\theta(y^{l})$ are distinct for all $l$, then 
  \[
  d_{l}(x^{l}, y^{l}) \geq d_{\T}( \theta(x^{l}), \theta(y^{l})) 
  \geq 2^{-n(x)}
  \]
  which is a contradiction. It follows that, for some $l$, we have 
  $\theta(x) = \theta(x^{l})  = \theta(y^{l}) = \theta(y)$. Together with that 
  the fact that $\tau_{\xi}(x_{n}) = \tau_{\xi}(y_{n})$, this implies
  that $x_{n} = y_{n}$, for $ 1 \leq n \leq n(x)$.
  
  We note that for $l \geq l_{0}$, $\sigma^{n(x)}(x^{l}),\sigma^{n(x)}(y^{l})$
  are in $X_{l-1}^{+}$ and we have 
  \begin{eqnarray*}
  d_{l-1}(\sigma^{n(x)}(x^{l}),\sigma^{n(x)}(y^{l})) 
  & =  &  d_{G/ \xi}( \sigma^{n(x)}(x^{l}),\sigma^{n(x)}(y^{l})) \\
    &  & 
    + \lambda_{l-1}(\sigma^{n(x)}(x^{l}),\sigma^{n(x)}(y^{l})) ) \\
     &  =  &  2^{n(x)} d_{G/ \xi}( x^{l},y^{l}) 
    + \lambda_{l-1}(\sigma^{n(x)}(x^{l}),\sigma^{n(x)}(y^{l})) ) \\
     &  \leq  & 2^{n(x)} d_{l}(x^{l}, y^{l}).
     \end{eqnarray*}
  After re-indexing the sequence, we can apply the same argument to 
  \newline
  $\sigma^{n(x)}(x), \sigma^{n(x)}(x), \sigma^{n(x)}(x^{l-1}), \sigma^{n(x)}(y^{l-1})$
  which shows that $x_{n} = y_{n}$, for \newline
  $ 1 \leq n \leq n(x) +n(\sigma^{n(x)})$.
  Continuing in this way, we see that $x=y$, as desired.
  \end{proof}

  \begin{cor}
  \label{constr:170}
  The space topological $X_{\xi}^{+}$ has a metric, $d_{\xi}$, 
  inducing the topology such that
  \begin{enumerate}
\item $d_{G_{ \xi}}(\tau_{\xi}(x), \tau_{\xi}(y))
 \leq d_{\xi}(\pi_{\xi}(x),\pi_{\xi}(y)) \leq 3 d_{G}(x,y)$,
 for all $x, y$ in $X_{G}^{+}$,
\item $ d_{\xi}(\pi_{\xi}(x),\pi_{\xi}(y)) = d_{k}(x,y)$, for all $x, y$ in $X_{k}^{+}$.
\end{enumerate}
  \end{cor}

  \begin{thm}
  \label{constr:180}
  \begin{enumerate}
  \item 
  The map $\rho_{\xi}:X_{\xi}^{+} \rightarrow X^{+}_{G_{\xi}}$ 
  satisfies 
  \[
  d_{G_{\xi}}(\rho_{\xi}(x), \rho_{\xi}(y)) \leq d_{\xi}(x,y),
  \]
  for all $x, y$ in $X_{\xi}^{+}$.
  \item 
  The map $\theta: X_{G}^{+} \rightarrow \T$ is constant
  on $\sim_{\xi}$-equivalence classes and so there is  a map 
  (also) denoted by $\theta : X^{+}_{\xi} \rightarrow \T$
  which satisfies
  \[
  d_{\T}(\theta(x), \theta(y))  \leq d_{\xi}(x,y),
  \]
  for all $x, y$ in $X_{\xi}^{+}$.
  \end{enumerate}
  \end{thm}
  
  \begin{proof}
  It is a trivial consequence of the definition of $d_{k}$
  that the inequality holds with $d_{k}$ on the right hand side
  for $x,y$ in $X_{k}^{+}$, for any $k \geq 1$. The conclusion holds 
  since $d_{\xi}$ is the common extension of all $d_{k}$.
  
  For any $k \geq 1$ and $x,y$ in $X_{k}^{+}$, we have 
  \[
  d_{\T}(\theta(x), \theta(y)) \leq \lambda_{k}(x,y) \leq d_{k}(x,y).
  \]
  Again, as $d_{\xi}$ is the common extension of all $d_{k}$, this
  also holds for $d_{\xi}$ on $X_{\xi}$.
  \end{proof}
  
  As we mentioned earlier, it is an easy matter to see that $X_{\xi}^{+}$
  is metrizable. We have gone to considerable trouble to 
  actually produce a metric. The reason is that it satisfies some nice properties
  and the following is the most important. We remark that the same property holds
  for the map $\sigma$ on $(X_{G}^{+}, d_{G})$.

  \begin{thm}
  \label{constr:190}
  If $x,y$ in $X^{+}_{\xi}$ satisfy $d_{\xi}(x,y) \leq 2^{-2}$, then 
  \[
2 d_{\xi}(x,y) \leq   d_{\xi}(\sigma_{\xi}(x), \sigma_{\xi}(y))
 \leq  8 d_{\xi}(x,y).
  \]
  \end{thm}
  
  \begin{proof}
  We will prove the analogous statement for $d_{k}, \sigma$ and 
   $x,y$ in $X^{+}_{k}$, for some
   fixed $k \geq 0$. 
   
   We first consider $k=0$ and note that $\sigma(X_{0}^{+}) = X_{0}^{+}$.
   For $x,y$ in $X_{0}^{+}$, we have 
   \[
   d_{0}(x,y) = d_{G_{\xi}}(\tau_{\xi}(x), \tau_{\xi}(y)) 
   + d_{\T}(\theta(x), \theta(y)).
   \]
   Our hypothesis implies both terms on the right are at most $2^{-2}$.
   This implies that $\tau_{\xi}(x_{1}) = \tau_{\xi}(y_{1})$ and hence 
   \[
   d_{G_{\xi}}(\tau_{\xi}(\sigma(x)), \tau_{\xi}(\sigma(y)))
   = d_{G_{\xi}}(\sigma\circ \tau_{\xi}(x), \sigma \circ \tau_{\xi}(y))  
   = 2 d_{G_{\xi}}(\tau_{\xi}(x), \tau_{\xi}(y))
   \]
   and $d_{\T}(\theta(x), \theta(y)) \leq 2^{-2}$ and hence
   \[
   d_{\T}(\theta(\sigma(x)), \theta(\sigma(y))) 
   = d_{\T}(\theta(x)^{2}, \theta(y)^{2}) = 2 d_{\T}(\theta(x), \theta(y)).
   \]
   It follows that $d_{0}(\sigma(x), \sigma(y)) = 2 d_{0}(x,y)$
    and we are done.
     
  Now assume $k \geq 1$.
  First, we note that $d_{k}(x,y) \leq 2^{-2}$ implies that 
   $d_{G_{\xi}}(\tau_{\xi}(x),\tau_{\xi}(y)) \leq 2^{-2}$ so 
   $\tau_{\xi}(x_{i}) = \tau_{\xi}(y_{i})$, for $i = 1, 2$.
   It follows that $n(x) = n(y)$ or both are at least $3$.
   
   Our first case is $n(x) =1 = n(y)$. It follows that 
   $\sigma(x), \sigma(y)$ 
   are both in $X^{+}_{k-1}$.
    We have $n(x) =1= n(y)$ and $\theta(x) = \theta(y)=1$ so
   \begin{eqnarray*}
   d_{\xi}(x,y) & = & d_{k}(x,y) \\
  & =  &   d_{G_{\xi}}(\tau_{\xi}(x), \tau_{\xi}(y)) + \lambda_{k}(x,y)\\
       &  =  &   d_{G_{\xi}}(\tau_{\xi}(x), \tau_{\xi}(y)) + 
       2^{-3} \lambda_{k-1}(\sigma(x), \sigma(y)) \\
       &  = &  2^{-1} d_{G_{\xi}}(\tau_{\xi}(\sigma(x)), \tau_{\xi}(\sigma(y))) + 
       2^{-3} \lambda_{k-1}(\sigma(x), \sigma(y)). 
       \end{eqnarray*}
       The right hand side is clearly bounded above by 
       \[
     2^{-1} d_{G_{\xi}}(\tau_{\xi}(\sigma(x)), \tau_{\xi}(\sigma(y))) + 
       2^{-1} \lambda_{k-1}(\sigma(x), \sigma(y))  
       =    2^{-1} d_{k-1}(\sigma(x),\sigma(y))
       \]
       and below by 
       \[
     8^{-1} d_{G_{\xi}}(\tau_{\xi}(\sigma(x)), \tau_{\xi}(\sigma(y))) + 
       8^{-1} \lambda_{k-1}(\sigma(x), \sigma(y)) 
        =    8^{-1} d_{k-1}(\sigma(x),\sigma(y)).
       \]

   Next, we consider $n(x) = n(y) = 2$. 
   Also, $\theta(x), \theta(y)$ are both $ \pm 1 $.
  If they are distinct, we have 
   \[
   d_{\xi}(x,y) \geq \lambda_{k}(x,y) \geq d_{\T}(\theta(x), \theta(y))
    \geq 2^{-1},
   \]
   which is a contradiction. Hence, we have $\theta(x)= \theta(y)$. So we 
   have $x, y$ are still in $X^{+}_{k}$, 
   $n(\sigma(x)) = n(x)-1 = n(\sigma(y))$  and $\theta(\sigma(x))=
   \theta(x)^{2} = \theta(y)^{2} = \theta(\sigma(y))$ and 
  \begin{eqnarray*}
   d_{\xi}(x,y) & = & d_{k}(x,y) \\
   & = &  d_{G_{\xi}}(\tau_{\xi}(x), \tau_{\xi}(y)) + \lambda_{k}(x,y)\\
       &  = &  2^{-1} d_{G_{\xi}}(\tau_{\xi}(\sigma(x)), \tau_{\xi}(\sigma(y))) + 
       2^{-n(x)-2} \lambda_{k-1}(\sigma^{n(x)}(x), \sigma^{n(y)}(y)) \\
       & =  & 2^{-1} \left( d_{G_{\xi}}(\tau_{\xi}(\sigma^{n(x)}(x)), 
       \tau_{\xi}(\sigma^{n(x)}(y))) \right. \\
        &  &  \left.
     +   2^{-1}2^{-n(x)-1}\lambda_{k-1}(\sigma(\sigma(x)), \sigma(\sigma(y))) \right) \\
      & =  & 2^{-1} \left( d_{G_{\xi}}(\tau_{\xi}(\sigma^{n(x)}(x)), 
       \tau_{\xi}(\sigma^{n(x)}(y))) \right.  \\
         &  & 
     + \left. 2^{-1} \lambda_{k}(\sigma(x), \sigma(y)) \right) \\
       &  =  &  2^{-1} d_{k}(\sigma(x),\sigma(y)).
   \end{eqnarray*}  
   
   Finally, we consider the case $n(x), n(y) > 2$.
   Here again, we have \newline
    $n(\sigma(x)) = n(x)-1$, $n(\sigma(y)) = n(y)-1$, 
   $\theta(\sigma(x)) = \theta(x)^{2}$ and 
   $\theta(\sigma(y)) = \theta(y)^{2}$.
   Let us first assume that $(n(x), \theta(x)) \neq (n(y), \theta(y))$.
   We claim that $(n(\sigma(x)), \theta(\sigma(x))) 
   \neq (n(\sigma(y)), \theta(\sigma(y)))$. If they are equal then 
   $n(x)-1 = n(y)-1$ so $n(x) = n(y)$. It follows that 
   $\theta(x) \neq \theta(y)$, while $\theta(x)^{2} = \theta(y)^{2}$
   which implies $\theta(x) = - \theta(y)$. This in turn 
   implies that
    $d_{k}(x,y) \geq d_{\T}(\theta(x), \theta(y)) = 2^{-1}$ which is a 
    contradiction. From the fact that 
     $2^{-2} \geq d_{k}(x,y) \geq d_{\T}(\theta(x), \theta(y)) $, we see that 
     $d_{\T}(\theta(x)^{2}, \theta(y)^{2}) = 2 d_{\T}(\theta(x), \theta(y)) $.
    
    Now we have
     \begin{eqnarray*}
   d_{\xi}(x,y) & = & d_{k}(x,y) \\
   & = &  d_{G_{\xi}}(\tau_{\xi}(x), \tau_{\xi}(y)) + \lambda_{k}(x,y)\\
       &  =  &   d_{G_{\xi}}(\tau_{\xi}(x), \tau_{\xi}(y)) + 
      \vert 2^{-n(x)} - 2^{-n(y)} \vert 
        +  d_{\T}(\theta(x), \theta(y)) \\
        & = & 2^{-1} d_{G_{\xi}}(\tau_{\xi}(\sigma(x)), \tau_{\xi}(\sigma(y))) 
    +  2^{-1} \vert 2^{1-n(x)} - 2^{1-n(y)} \vert  \\
    &  &          + 2^{-1} d_{\T}(\theta(x)^{2}, \theta(y)^{2}) \\
       &  =  &  2^{-1} d_{k}(\sigma(x),\sigma(y)) \\
       & = &  2^{-1} d_{\xi}(\sigma(x),\sigma(y)).
   \end{eqnarray*}

   The final case is $(n(x), \theta(x)) =(n(y), \theta(y))$. Here, we have
  \newline 
   $(n(\sigma(x)), \sigma(\theta(x)) =(n(x)-1, \theta(x)^{2})= (n(\sigma(y)), 
   \sigma(y))$
   and 
     \begin{eqnarray*}
   d_{\xi}(x,y) & = & d_{k}(x,y) \\
   & = &  d_{G_{\xi}}(\tau_{\xi}(x), \tau_{\xi}(y)) + \lambda_{k}(x,y)\\
       &  =  &   d_{G_{\xi}}(\tau_{\xi}(x), \tau_{\xi}(y)) + 
     \lambda_{k}(x,y) \\
       &  =  &   d_{G_{\xi}}(\tau_{\xi}(x), \tau_{\xi}(y)) + 
     2^{-n(x)-2} \lambda_{k-1}(\sigma^{n(x)}(x),\sigma^{n(y)}(y)) \\
           &  =  & 2^{-1} \left(  
            d_{G_{\xi}}(\tau_{\xi}(\sigma(x)), \tau_{\xi}(\sigma(y))) \right.\\
  &  &  + \left.  
     2^{1-n(x)-2} \lambda_{k-1}( 
    \sigma^{n(x)-1}(\sigma(x)),\sigma^{n(y)-1}(\sigma(y)) ) \right) \\
      &  =  & 2^{-1} \left(   d_{G_{\xi}}(\tau_{\xi}(\sigma(x)), 
      \tau_{\xi}(\sigma(y)))  +   \lambda_{k}(\sigma(x),\sigma(y)) \right) \\
     &  =  &  2^{-1} d_{k}(\sigma(x),\sigma(y)) \\ 
       & = &  2^{-1} d_{\xi}(\sigma(x),\sigma(y)).
   \end{eqnarray*} 
  \end{proof}

We finish this section with an alternate version of the 
local expanding property which will be useful in the later sections.

\begin{lemma}
\label{constr:200}
If $x, y$ are in $X_{G}^{+}$ satisfy $x_{1} = y_{1}$ and 
$d_{\xi}(\sigma(x), \sigma(y)) \leq 2^{-1}$, then 
$d_{\xi}(x,y) \leq 2^{-1} d_{\xi}(\sigma(x), \sigma(y))$. 
\end{lemma}

\begin{proof}
By Lemma \ref{constr:50}, we may find sequences
 $x^{k}, k \geq 0,$ and $y^{k}, k \geq 0$, 
converging to $x$ and $y$ respectively, and $x^{k}, y^{k}$ are in $X_{k}^{+}$.
So it suffices to prove the statement for $x,y$ in $X_{k}^{+}$, for $k \geq 2$.

We consider three cases separately. The first is when $x_{1} = y_{1}$ is 
not in $\xi(H^{1})$.  This implies that $n(x) = n(y) = 1$,
 $\theta(x) = \theta(y) =1$
 and 
 $n( \sigma(x)) = n(\sigma(y)) = n(x) -1$. It also means that 
 $\sigma(x), \sigma(y)$ are in $X_{k-1}^{+}$. It follows 
 \begin{eqnarray*}
 d_{\xi}(\sigma(x), \sigma(y)) & = & 
 d_{G_{\xi}}(\tau_{\xi}(\sigma(x)), \tau_{\xi}(\sigma(y))) +
    \lambda_{k-1}(\sigma(x), \sigma(y)) \\
    &  =  &  2 d_{G_{\xi}}(\tau_{\xi}(x), \tau_{\xi}(y)) + 2^{3} \lambda_{k}(x,y) \\
      &  \geq & 2 d_{\xi}(x, y).
 \end{eqnarray*}
 
 The second case is $x_{1} = y_{1}$ is 
in $\xi(H^{1})$, which implies that $n(x),  n(y) > 1$,
and  $(n(x),\theta(x)) \neq (n(y), \theta(y) )$. 
 We have $n( \sigma(x))= n(x) -1$ and  $n(\sigma(y)) = n(y) -1$, 
 $\theta(\sigma(x))= \theta(x)^{2}, \theta(\sigma(y))= \theta(y)^{2}$.
 It also means that 
$\sigma(x), \sigma(y)$ are in $X_{k}^{+}$.
  We have 
  \begin{eqnarray*}
 d_{\xi}(\sigma(x), \sigma(y)) & = & 
 d_{G_{\xi}}(\tau_{\xi}(\sigma(x)), \tau_{\xi}(\sigma(y))) +
    \lambda_{k}(\sigma(x), \sigma(y)) \\
    &  =  &  2 d_{G_{\xi}}(\tau_{\xi}(x), \tau_{\xi}(y)) 
    + \vert 2^{-n(x)+1} - 2^{-n(y)+1} \vert  \\
     &  & 
      +  d_{\T}(\theta(x)^{2}, \theta(y)^{2}) \\
        &  =  &  2 d_{G_{\xi}}(\tau_{\xi}(x), \tau_{\xi}(y)) 
        + 2 \vert 2^{-n(x)} - 2^{-n(y)} \vert  \\
          &  & 
      +  d_{\T}(\theta(x)^{2}, \theta(y)^{2}). \\     
 \end{eqnarray*}
The hypothesis that  $d_{\xi}(\sigma(x), \sigma(y)) \leq  2^{-1}$ implies that 
 $d_{\T}(\theta(x)^{2}, \theta(y)^{2}) \leq 2^{-1}$. The fact that $x_{1}=y_{1}$ 
 implies $\varepsilon(x_{1}) = \varepsilon(y_{1})$. Together these
  imply that
  $ d_{\T}(\theta(x)^{2}, \theta(y)^{2}) = 2d_{\T}(\theta(x), \theta(y)) $.
  Putting this in the last line above yields
  $ d_{\xi}(\sigma(x), \sigma(y))  = 2 d_{\xi}(x,y)$.
  
  The final case to consider is that $x_{1} = y_{1}$ is 
in $\xi(H^{1})$, which implies that $n(x),  n(y) > 1$,
and  $(n(x),\theta(x)) =(n(y), \theta(y) )$. Again, we know 
$\sigma(x), \sigma(y)$ are  in $X_{k}^{+}$.
Here we have $n( \sigma(x))= n(x) -1 = n(\sigma(y)) = n(y) -1$ and 
 $\theta(\sigma(x))= \theta(x)^{2} = \theta(\sigma(y))= \theta(y)^{2}$
 We have 
 \begin{eqnarray*}
 d_{\xi}(\sigma(x), \sigma(y)) & = & 
 d_{G_{\xi}}(\tau_{\xi}(\sigma(x)), \tau_{\xi}(\sigma(y))) +
    \lambda_{k}(\sigma(x), \sigma(y)) \\
    &  =  &  2 d_{G_{\xi}}(\tau_{\xi}(x), \tau_{\xi}(y)) + 2^{-2 -( n(x) -1)} \\
     &  &  
     \lambda_{k-1}(\sigma^{n(x)-1}(\sigma(x)),
   \sigma^{n(y)-1}(\sigma(  y)) \\  
     &  =  &  2 \left( d_{G_{\xi}}(\tau_{\xi}(x), \tau_{\xi}(y)) + 2^{-2 - n(x) }
     \lambda_{k}(\sigma^{n(x)}(x),
   \sigma^{n(y)}( y)) \right)\\ 
      &  = & 2 d_{\xi}(x, y).
 \end{eqnarray*}
\end{proof}

  \begin{thm}
  \label{constr:210}
  If $x, y$ are in $X_{\xi}^{+}$ with $d_{\xi}(x, \sigma_{\xi}(y)) \leq 2^{-1}$, 
  then there is $z$ in $X_{\xi}^{+}$ with 
  $d_{\xi}(z,y) \leq 2^{-1} d_{\xi}(x, \sigma_{\xi}(y))$ and $\sigma_{\xi}(z) = x$.
  In particular,
  we
  have 
  \[
  X_{\xi}^{+}(\sigma_{\xi}(x), \epsilon) \subseteq 
  \sigma_{\xi}(X_{\xi}^{+}(x, 2^{-1} \epsilon))
  \]
 and the map $\sigma_{\xi}$ is open.
  \end{thm}
  
  \begin{proof} 
Write $x = \pi_{\xi}(x')$, for some $x'$ in $X_{G}^{+}$. Let $y'$ be in 
  $X_{G}^{+}$ with $\pi_{\xi}(y') =y$.
  As $\epsilon \leq 2^{-3}$, we know that    
  $d_{G_{\xi}}(\tau_{\xi}(x'), \tau_{\xi}(\sigma(y'))) \leq 2^{-3}$ also. 
  This implies that $\tau_{\xi}(x'_{1}) = \tau_{\xi}(y'_{2})$, which in turn 
  implies that $x'_{1} = y_{2}'$. It follows that 
  $z' = (y'_{1}, x'_{1}, x'_{2}, \ldots)$ is in $X_{G}^{+}$. 
  Applying Lemma \ref{constr:190} to $z', y'$ gives
  \[
  d_{\xi}(y,z) = d_{\xi}(y',z') \leq 2^{-1} d_{\xi}(\sigma(y'), \sigma(z')) 
    = 2^{-1} d_{\xi}(\sigma(y), x).
    \]
  \end{proof}

\section{Smale spaces}
\label{smale}

The results of section \ref{constr} showed a construction
of a compact metric space $(X_{\xi}^{+}, d_{\xi})$ along 
with a map $\sigma_{\xi}: X_{\xi}^{+} \rightarrow X_{\xi}^{+}$ 
which is continuous, surjective and open. It also satisfies
a local expansiveness condition Theorem \ref{constr:190}.
Our goal in this section is to replace this system 
with another where the dynamics is actually a homeomorphism.
This is a standard construction via inverse limits. Without giving
the precise definition for the moment, this introduces a new 
component to the space where the dynamics is contracting. In short, 
we are going to build a Smale space. Two references for Smale spaces
are Ruelle \cite{Ru:ThFor} and Putnam \cite{Pu:HSm}.

Let us begin with recalling the definition 
of a Smale space. First of all, $(X, d)$ is a compact metric space 
and $\varphi$ is a homeomorphism  of $X$. We assume that there are 
constants $\epsilon_{X} > 0, 0 < \lambda < 1$ and a 
continuous map
\[
[ \cdot, \cdot] : \{ (x, y) \mid x,y \in X, d(x,y) \leq \epsilon_{X} \}
 \rightarrow X,
\]
which satisfies the following:

\begin{enumerate}
\item[B1] $[x,x] = x$, 
\item[B2] $ [x,[y,z]] = [x,z]$,
\item[B3] $[[x,y],z] = [x,z]$,
\item[B4] $[ \varphi(x), \varphi(y)] = \varphi[x,y]$,
\end{enumerate}
for all $x, y, z$ in $X$, whenever both sides of an equation 
are defined,
and finally
\begin{enumerate}
\item[C1] if $[x,y] = y$, then $d(\varphi(x), \varphi(y)) \leq \lambda d(x,y)$,
\item[C2] if $[x,y] = x$, then $d(\varphi^{-1}(x), \varphi^{-1}(y)) 
\leq \lambda d(x,y)$,
\end{enumerate}
for $x,y$ with $d(x,y) \leq \epsilon_{X}$.

We say that $(X, d, \varphi)$ is a Smale space. We note that it is not necessary to
specify the bracket map as part of the structure,
 only its existence: if it exists, it is essentially unique.
 
We now consider our specific case of interest. We will usually work 
under the standing hypotheses. It seems likely that it is sufficient for 
$G$ to be irreducible. That is only used in 
an essential way in Proposition \ref{constr:50}, but it seems possible
that some other hypothesis on $\xi$ may be needed for the construction
of the last section to work.

\begin{defn}
\label{smale:10}
Let $G, H, \xi$ satisfy the standing hypotheses.
 We define $X_{\xi}$ to be the inverse limit
of the system 

\vspace{.5cm}

\hspace{2cm}
\xymatrix{
X_{\xi}^{+}  & X_{\xi}^{+} \ar_{\sigma_{\xi}}[l] & \ar_{\sigma_{\xi}}[l] & \cdots }

\vspace{.5cm}

\noindent That is, 
\[
X_{\xi} = \{ (x^{0}, x^{1}, \ldots ) \mid x^{n} \in X^{+}_{\xi}, 
x^{n} = \sigma_{\xi}(x^{n+1}), n \geq 0 \}.
\]
We also define $\sigma_{\xi}: X_{\xi} \rightarrow X_{\xi}$ by 
$\sigma_{\xi}(x)^{n} = \sigma_{\xi}(x^{n})$, for $n \geq 0$ and 
$x^{n}, n \geq 1$ in $X_{\xi}$. It is a simple matter to check that 
the inverse is given by 
\[
\sigma_{\xi}^{-1}(x)^{n} = x^{n+1}, n \geq 0, 
\]for any $x$ in $X_{\xi}$.

Finally, we define a metric, also denoted $d_{\xi}$, on $X_{\xi}$ by 
\[
d_{\xi}(x, y) = \sup\{ 2^{-n} d_{\xi}(x^{n}, y^{n})  \mid n \geq 0 \},
\]
for $x= (x^{n})_{ n \geq 0}, y = (y^{n})_{n \geq 0}$ in $X_{\xi}$.
\end{defn}

We remark that if we replace $X_{\xi}^{+}, \sigma_{\xi}$ 
by $X_{G}^{+}, \sigma_{G}$
the result is $X_{G}, \sigma_{G}$ in a canonical way 
by identifying $x$ in $X_{G}$ with the sequence
$x^{n} = \chi_{G}^{+}(\sigma^{-n}(x)) = (x_{1-n}, x_{2-n}, \ldots)$, for $ n \geq 0$.
Pursuing this remark a little further, it is a fairly simple matter to check that
the map which sends $x$ in $X_{G}$ to the sequence 
$\pi_{\xi}(\chi_{G}^{+}(\sigma^{-n}(x))), n \geq 0,$
in $X_{\xi}$ is well-defined. We denote the map by $\pi_{\xi}$.
 It is immediate that 
$\pi_{\xi} \circ \sigma  =  \sigma_{\xi}   \circ \pi_{\xi} $.

The system $(X_{G}, d_{G}, \sigma)$ is a Smale space
with constants $\epsilon_{X_{G}} = \lambda = 2^{-1}$ and bracket defined as
$[x,y] = ( \ldots y_{-1}, y_{0}, x_{1}, \ldots )$, for
$d_{G}(x,y) \leq 2^{-1}$, which ensures $x_{0}=y_{0}, x_{1} = y_{1}$.

To see that $(X_{\xi}, d_{\xi}, \sigma_{\xi})$ is a Smale space, we need
to define the bracket map. Let 
 $x= (x^{n})_{ n \geq 0}, y = (y^{n})_{n \geq 0}$ be in $X_{\xi}$ and 
 assume that $d_{\xi}(x, y) \leq 2^{-n-1}$. 

We inductively define $z^{n}, n = 0,  1, 2, 3,  \ldots$ in $X_{\xi}^{+}$ satisfying 
\begin{enumerate}
\item 
$d_{\xi}(z^{n}, y^{n}) \leq 2^{-n-1},$
\item 
$\sigma_{\xi}(z^{n+1}) = z^{n}$.
\end{enumerate}

 Begin with  $z^{0} = x^{0}$.
 Observe the first condition holds for $n=0$ 
 and 
 \[
 d_{\xi}(x^{0}, y^{0}) \leq d_{\xi}(x, y)  \leq 2^{-1}.
 \]
 
 To define $z^{n+1}$, we have  
 $d_{\xi}(z^{n}, \sigma_{\xi}(y^{n+1})) = d_{\xi}(z^{n},y^{n}) \leq 2^{-n-1}$.
 By Lemma \ref{constr:200}, we may find $z^{n+1}$ in $X_{\xi}^{+}$ with 
 $\sigma_{\xi}(z^{n+1})=z^{n}$ and 
 $d_{\xi}(z^{n+1}, y^{n+1}) \leq 2^{-1} d_{\xi}(z^{n},y^{n}) \leq 2^{-n-2}$.
 So $z = (z^{0}, z^{1}, \ldots)$ lies in $X_{\xi}$.
 
 We use the constants $\epsilon_{X_{\xi}} =  \lambda = 2^{-1}$. 
 Notice that $d_{\xi}(x,y) \leq 2^{-1}$
  implies that $d_{\xi}(x^{0}, y^{0}) \leq 2^{-1}$, 
 which ensures that $[x,y]$ is defined.
 
 We will not verify the axioms B1-B4. As for C1, suppose that $z=[x,y]=y$. 
 It follows that $y^{0}=x^{0}$ and so 
 \begin{eqnarray*}
 d(\sigma_{\xi}(x), \sigma_{\xi}(y)) & = &  \sup \{ 
 2^{-n}d_{\xi}(\sigma_{\xi}(x^{n}), \sigma_{\xi}(y^{n})) \mid n \geq 0 \} \\
   & = &  \sup \{ 
 2^{-n}d_{\xi}(\sigma_{\xi}(x^{n}), \sigma_{\xi}(y^{n}) ) \mid n \geq 1 \} \\
   & = &  \sup \{ 
2^{-n}d_{\xi}(x^{n-1}, y^{n-1}) \mid n \geq 1 \} \\
    & = & 2^{-1} \sup \{ 
2^{-m}d_{\xi}(x^{m}, y^{m}) \mid m \geq 0 \} \\
   & =  & 2^{-1} d_{\xi}(x, y).
   \end{eqnarray*}
   
  We also verify C2:  if $z = [x,y] = x$, then we have 
  \begin{eqnarray*}
 d(\sigma_{\xi}^{-1}(x), \sigma_{\xi}^{-1}(y)) & = &  \sup \{ 
 2^{-n}d_{\xi}(x^{n+1}, y^{n+1}) \mid n \geq 0 \} \\ 
    & = &  \sup \{ 
 2^{-n} 2^{-1}d_{\xi}(x^{n}, y^{n}) \mid n \geq 0 \} \\ 
   &  =  &  2^{-1} d_{\xi}(x,y).
   \end{eqnarray*}

\begin{thm}
\label{smale:30}
With constants $\epsilon_{X_{\xi}} =  \lambda= 2^{-1}$
and bracket map as defined above, 
 $(X_{\xi}, d_{\xi}, \sigma_{\xi})$ is a Smale space.
\end{thm}

\begin{thm}
\label{smale:50}
\begin{enumerate}
\item
The map $\pi_{\xi}: (X_{G}, d_{G}, \sigma) 
\rightarrow (X_{\xi}, d_{\xi}, \sigma_{\xi})$
is a factor map.
\item 
For each $x, y$ in $X_{G}$ such that $d_{G}(x,y) \leq 2^{-3}$, both 
\newline
$[x,y]$ and $[\pi_{\xi}(x), \pi_{\xi}(y)]$ are defined and 
$\pi_{\xi}[x,y] = [\pi_{\xi}(x), \pi_{\xi}(y)]$.
 \item Points $x, y$  in $X_{G}$ satisfy $\pi_{\xi}(x) = \pi_{\xi}(y)$ 
 if and only if 
 exactly one of the following hold:
 \begin{enumerate}
 \item $x=y$, 
 \item there exists $i=0,1$ and $z$ in $X_{H}$ such that 
 $\xi^{i}(z_{n}) = x_{n},  \xi^{1-i}(z_{n} ) =y_{n}$, for all 
 integers $n$,
 \item there is an integer $m$, $i=0,1$  and $z$ in $X_{H}$ such that
 \begin{enumerate}
 \item
 $x_{n} = y_{n}$, for  all
 $n < m$,  
 \item $x_{m} = y_{m}$ is not in $\xi(H)$ or
 $\xi^{1-i}(z_{m}) = x_{m},  \xi^{i}(z_{m} ) =y_{m}$ and
 \item 
  $\xi^{i}(z_{n}) = x_{n},  \xi^{1-i}(z_{n} ) =y_{n}$, for all 
 integers $n> m$.
 \end{enumerate}
  \end{enumerate} 
 \item  
 The map $\pi_{\xi}$ is $s$-bijective; that is, for every $x$ 
in $X_{G}$, $\pi_{\xi}| X^{s}_{G}(x)$ is a bijection
from  $  X^{s}_{G}(x)$  to  $  X^{s}_{\xi}(\pi_{\xi}(x))$.
\end{enumerate}
\end{thm}

\begin{proof}
The first part is easy, as we have already noted.

For the second part, if $d_{G}(x,y) \leq 2^{-3}$, then $[x,y]$ is defined. 
In addition, by \ref{constr:170}, 
$d_{\xi}(\pi_{\xi}(x), \pi_{\xi}(y)) \leq 3 \cdot 2^{-3} < 2^{-1}$, 
so $[\pi_{\xi}(x), \pi_{\xi}(y)]$ is also defined. 

We define, for $n \geq 0$, 
$z^{n} = \pi_{\xi}(y_{-n}, \cdots, y_{-1}, x_{0}, x_{1}, \ldots)$ in $X_{\xi}^{+}$. 
It is clear that $z^{0} = \pi_{\xi}(\chi_{G}^{+}(x))$ and 
$\sigma_{\xi}(z^{n}) = z^{n-1}$, for all $n \geq 1$. We claim that
$d_{\xi}(z^{n}, y^{n}) \leq 2^{-n-1}$, which we show by induction.
The case $n=0$ is simply the fact that 
$d_{\xi}(\pi_{\xi}(x), \pi_{\xi}(y)) \leq 2^{-1}$,
which we have already established. For any $n \geq 1$, Lemma \ref{constr:200} 
and the induction hypothesis show that
\[
d_{\xi}(z^{n}, y^{n}) \leq 2^{-1} d_{\xi}(z^{n-1}, y^{n-1}) \leq 2^{-n-1}.
\]
It follows from the definition of the bracket on $X_{\xi}$ that 
$z =  [\pi_{\xi}(x), \pi_{\xi}(y)]$. On the other hand, we have 
\[
(\pi_{\xi}[x,y])^{n} = \pi_{\xi}([x,y]^{n})= 
\pi_{\xi}(\chi_{G}^{+}(\sigma^{-n}[x,y]) ) = z^{n},
\]
by definition.

For the third part, it is clear that if $x, y$ satisfy
 conditions (b) or (c), then 
$ \chi_{G}^{+}(\sigma^{-n}(x)) \sim_{\xi} \chi_{G}^{+}(\sigma(y))$,
 for all integers $n$ and it follows
that $\pi_{\xi}(x)^{n} = \pi_{\xi}(y)^{n}$, for all integers $n$. 

Conversely, if $\pi_{\xi}(x) = \pi_{\xi}(y)$, then 
$\chi_{G}^{+}(\sigma^{-n}(x)) \sim_{\xi} \chi_{G}^{+}(\sigma^{-n}(y))$, 
for all integers $n$.
It follows that there is an integer $n$ such that, there is  a path
$z^{n}$ in $X_{H}^{+}$ and $i_{n}=0,1$ such that 
\[
\chi_{G}^{+}(\sigma^{1-n}(x) = \xi^{i_{n}}(\chi_{G}^{+}(\sigma^{1-n}(z^{n})), 
\chi_{G}^{+}(\sigma^{1-n}(y) = \xi^{1-i_{n}}(\chi_{G}^{+}(z^{n})).
\]
If $n$ satisfies this condition, then so does $n+1$ and $i_{n+1}=i_{n}$, 
$\chi_{G}^{+}(\sigma^{n}(z^{n+1})= \chi_{G}^{+}(\sigma^{1-n}(z^{n})$.
 There are then two cases to consider.
In the first, the condition is satisfied by all integers $n$. 
In this case, (b) holds with $z = \lim_{n \rightarrow - \infty} z^{n}$ (and 
possibly) reversing $x$ and $y$ if $i_{m}=1$. Otherwise, let $m$ be the greatest
integer not satisfying the condition. It is a simple matter to check that 
(c) holds in this case.

For the last part, if $\pi_{\xi}(x) = \pi_{\xi}(y)$, then by part 3, $x$ and $y$ are
not stably equivalent (i.e. right tail-equivalent) unless $x=y$. This shows that 
$\pi_{\xi}| X^{s}_{G}(x)$ is injective. The surjectivity follows from  the fact that $G$ primitive implies $(X_{G}, \sigma)$ is non-wandering, 
(Proposition 2.2.14 of \cite{LM:book}), 
and 
Theorem 2.5.8 
of \cite{Pu:HSm}.
\end{proof}

\section{Homology}
\label{hom}

In \cite{Pu:HSm}, a homology theory for (non-wandering) 
Smale 
spaces is described. Specifically, given 
a non-wandering Smale space $(X, d, \varphi)$
   and integer $n$, there are 
two countable abelian 
groups, $H^{s}_{n}(X, \varphi)$ and $H^{u}_{n}(X, \varphi)$. The former 
invariant is covariant for 
$s$-bijective factor maps and contravariant for $u$-bijective factor maps
while the latter is covariant for 
$u$-bijective factor maps and contravariant for $s$-bijective factor maps.
As such the map $\varphi$ induces a pair of automorphisms of each.
 The aim of this section is to compute these invariants for the systems 
 $(X_{\xi}, d_{\xi}, \sigma_{\xi})$ we have constructed.

  Let us begin with some preliminary notions.
  If $G$ is any finite directed graph, we may consider
 the free abelian group on 
 the vertex set $G^{0}$, denoted $\Z G^{0}$. This comes with 
 two canonical endomorphisms
 \[
 \gamma_{G}^{s}(v) =  \sum_{i(e) = v} t(e), \hspace{1cm}
 \gamma_{G}^{u}(v)  =  \sum_{t(e) = v} i(e),
 \]
 for any $v$ in $G^{0}$. 
 This is needed at various places
 in the theory and in our proofs below. However, for the statements 
 of the results, we can  use the more familiar group $\Z^{G^{0}}$, 
 which is isomorphic to $\Z G^{0}$ in an obvious way. Under this isomorphism, the 
 map $\gamma_{G}^{s}$ becomes multiplication by the matrix $A_{G}$ while 
 $\gamma^{u}_{G}$ is multiplication by its transpose.
 Then we define $ D^{s}(G)$ as the inductive limit
 of the sequence
 \[
 \Z^{G^{0}} \stackrel{A_{G}}{\rightarrow} \Z^{G^{0}} \stackrel{A_{G}}{\rightarrow} 
 \Z^{G^{0}} \stackrel{A_{G}}{\rightarrow}  \cdots
 \]
 Similarly, the invariant $ D^{u}(G)$ is computed in a similar
 way, but with the transpose, $A_{G}^{T}$, replacing $A_{G}$.

  With this notation, we can summarize our results 
  with the following theorem.

 \begin{thm}
 \label{hom:20}
 Let $G, H, \xi$ satisfy the standing hypotheses.
 We have 
 \[
 \begin{array}{cccccc}
 H^{s}_{0}(X_{\xi}, \sigma_{\xi}) & \cong & D^{s}(G), &
 (\sigma_{\xi})_{*}^{-1} & \cong & A_{G}, \\
  H^{s}_{1}(X_{\xi}, \sigma_{\xi}) & \cong & D^{s}(H), &
 (\sigma_{\xi})_{*}^{-1} & \cong & A_{H},\\
   H^{u}_{0}(X_{\xi}, \sigma_{\xi}) & \cong & D^{u}(G),&
 (\sigma_{\xi})_{*} & \cong & A_{G}^{T}, \\
  H^{u}_{1}(X_{\xi}, \sigma_{\xi}) & \cong & D^{u}(H), &
 (\sigma_{\xi})_{*} & \cong & A_{H}^{T},\\
   H^{s}_{k}(X_{\xi}, \sigma_{\xi}) & \cong & 0, & k & \neq &  0,1,   \\
  H^{u}_{k}(X_{\xi}, \sigma_{\xi}) & \cong & 0,  & k &  \neq & 0,1.  
  \end{array}
  \]
  In  the first two lines, we regard $\sigma_{\xi}$ as an $s$-bijective
  factor map from $(X_{\xi}, \sigma_{\xi})$ to itself and 
  our description of the induced map on homology
  is interpreted via the description which precedes it. In the next two lines, 
  we regard $\sigma_{\xi}$ as a $u$-bijective
  factor map 
 \end{thm}
 
 In fact, we will prove only the first, second and fifth parts as the other 
 three are quite similar.

The starting point for computation of the homology theory for a general
Smale space, $(X, \varphi)$ is to find an $s/u$-bijective pair: Smale spaces
$(Y, \psi)$ and $(Z, \zeta)$ such that $Y^{u}(y)$ and 
$Z^{s}(z)$ are totally disconnected, 
for all $y$ in $Y$ and $z$ in $Z$, along with an $s$-bijective factor map
$\pi_{s}: (Y, \psi) \rightarrow (X, \varphi)$
and a  $u$-bijective factor map
$\pi_{u}: (Z, \zeta) \rightarrow (X, \varphi)$. Usually, the collection
$(Y, \psi, \pi_{s}, Z, \zeta, \pi_{u})$ is written as $\pi$.

Given such an $s/u$-bijective pair, one proceeds to define dynamical systems
as fibred products: for $L, M \geq 0$,
\begin{eqnarray*}
\Sigma_{L,M}(\pi) & = &  \{ (y^{0}, \ldots, y^{L}, z^{0}, \ldots, z^{M}) \in 
Y^{L+1} \times Z^{M+1}  \\
  &  &  \mid \pi_{s}(y^{l}) = \pi_{u}(z^{m}), \text{ all } l, m \}.
\end{eqnarray*}
Each of these systems is a shift of finite type and $\Sigma_{L,M}(\pi)$ carries
an obvious action of the group $S_{L+1} \times S_{M+1}$.

\begin{thm}
\label{hom:30}
 Let $G, H, \xi$ satisfy the standing hypotheses. 
Then  \newline
$(X_{G}, \sigma, \pi_{\xi}, X_{\xi}, \sigma_{\xi}, id_{X_{\xi}})$ 
is an $s/u$-bijective pair for $(X_{\xi}, \sigma_{\xi})$, which we 
denote by $\pi_{\xi}$. Moreover, we have 
$\Sigma_{0,0}(\pi_{\xi}) = X_{G}$ and 
$C^{s}_{\mathcal{Q}, \mathcal{A}}(\pi_{\xi})_{0,0} \cong D^{s}(G)$.
\end{thm}

The next step is to find a graph which   gives a symbolic presentation
for $\Sigma_{0,0}(\pi_{\xi}) $. 
This requires that, for any $x,y$ in $X_{g}$ satisfying
$t(x_{0}) = t(y_{0})$, then $\pi_{\xi}[x,y] = [ \pi_{\xi}(x), \pi_{\xi}(y)]$,
 meaning that 
both sides are defined.
Unfortunately, $G$ will not suffice
but this is
not a major obstacle.

We review the standard notion of the higher block coding of $G$ (see 
 section 2.3 \cite{LM:book}). 
 We let $G^{K}$ be the set of all paths of length $K$ in $G$, for any $K \geq 2$,
 and define 
 $i,t :G^{K} \rightarrow G^{K-1}$ by 
 \[
 i(x_{1}, \ldots, x_{K}) = (x_{1}, \ldots, x_{K-1}), \hspace{1cm}
   i(x_{1}, \ldots, x_{K}) = (x_{2}, \ldots, x_{K}),
   \] 
   for  $(x_{1}, \ldots, x_{K})$ in $G^{K}$. In this way, we interpret
   $G^{K},G^{K-1},i,t$ as a graph. 
   For any $1 \leq k \leq K$, 
   the map sending 
    $(x_{1}, \ldots, x_{K})$ in $G^{K}$ to $x_{k}$ induces 
    a homeomorphism between $X_{G^{K}}$ and $X_{G}$ which 
    intertwines  the shift maps. It will not be necessary for us to
    write this map explicitly: we simply regard $X_{G^{K}}$ and $X_{G}$
    as equal, in an obvious way.
    
    We use $K=7$ and $k=4$. The point is that,
    if $x,y$ are in $X_{G}$ and $t(x_{0}) = t(y_{0})$, 
  then  regarding them as elements in 
    $X_{G^{7}}$, means $x_{[-2,3]} = y_{[-2,3]}$, which implies 
    $d(x,y) \leq 2^{-3}$ and so part 2 of Theorem \ref{smale:50} applies
    and $G^{7}$ satisfies the conditions of Definition 2.6.8 of 
    \cite{Pu:HSm}.
    
    We remark that the map $t: G^{6} \rightarrow G^{0}$ induces an 
    isomorphism between $D^{s}(G^{7})$ and $D^{s}(G)$.

  Our two maps are $\pi_{s} = \pi_{\xi}$ and $\pi_{u}$ being 
the identity on $X_{\xi}$. Due to  part 3 of Theorem \ref{smale:50},
we see that $\# \pi_{s}\{ x \} \leq 2 =L_{0}$ and obviously
$\# id^{-1}\{ x \} \leq 1 = M_{0}$, for all $x$ in $X_{\xi}$, so 
by Theorem 5.1.10 of \cite{Pu:HSm}, our complex 
consists of only two non-zero groups,
$C^{s}_{\mathcal{Q}, \mathcal{A}}(\pi_{\xi})_{0,0}$ and 
$C^{s}_{\mathcal{Q}, \mathcal{A}}(\pi_{\xi})_{1,0}$, and
 a single homomorphism from the latter to the former. 
 So what are these groups, how do we compute them  and 
 what is the homomorphism between them? The first is given already in \ref{hom:30}.

 We now come to the second group,
  $C^{s}_{\mathcal{Q}, \mathcal{A}}(\pi_{\xi})_{1,0}$, 
 where it will be necessary to use our higher block presentation.
 The system $\Sigma_{(1,0)}(\pi_{\xi})$ of \cite{Pu:HSm} consists of pairs
 \[
 \Sigma_{(1,0)}(\pi_{\xi}) = 
 \{ (x^{0}, x^{1}) \mid 
  x^{0}, x^{1} \in X_{G}, \pi_{\xi}(x^{0}) = \pi_{\xi}(x^{1}) \}.
 \]
 Fortunately, Theorem \ref{smale:50} gives a nice description of these. The fact that 
 our map is regular means that this is the shift of finite type 
 associated with an obvious subgraph of $G^{7}_{1} \subseteq 
 G^{7} \times G^{7}$, whose vertex set is $G^{6}_{1} \subseteq  G^{6} \times G^{6}$, 
 with obvious maps $i,t$.  (This suppresses 
 the fact that an infinite sequence of pairs can also be seen as 
 a pair of infinite sequences.) This graph is obtained by simply taking pairs in 
 $\Sigma_{(1,0)}(\pi_{\xi})$ and finding all words of length $6$ and $7$.

We may partition our vertex set 
$G^{6}_{1} = V_{0} \cup V_{1} \cup \cdots  \cup V_{6}$
as follows. First,  $V_{0}$  consists of all pairs $(x,x)$, 
where $x$ is in $G^{6}$. Also, $V_{6}$ consists
of all pairs $(\xi^{0}(y), \xi^{1}(y)), (\xi^{1}(y), \xi^{0}(y))$, 
where $y$ is in $H^{6}$. For $ 1 \leq k < 6$, we let $V_{k}$ be all 
pairs of the following types
\[
(x\xi^{0}(y), x\xi^{1}(y)), (x\xi^{1}(y), x\xi^{0}(y)), 
\]
where, $x   $ is in $G^{n-k}$, $y$ is in  $H^{k}$ and
$t(g) = \xi^{0}(i(y))$, along with 
\begin{eqnarray*}
(x\xi^{1}(y_{1})\xi^{0}(y_{2}),\ldots  ,\xi^{0}(y_{k+1})),
 & & x\xi^{0}(y_{1})\xi^{1}(y_{2}), \ldots, \xi^{1}(y_{k+1]}))),\\
 (x\xi^{0}(y_{1})\xi^{1}(y_{2}), \ldots, \xi^{1}(y_{k+1})),  &  & 
 x\xi^{1}(y_{1})\xi^{0}(y_{2}), \ldots,  \xi^{0}(y_{k+1})),  
\end{eqnarray*}
 where   $ x $ is in $G^{n-k-1}$, $y$ is  in $ H^{k+1}$ and  
 $ t(g) = \xi^{0}(i(y))$. 

There is an analogous partition 
$G^{7}_{1} = E_{0} \cup \cdots \cup E_{7}$.
Notice that  $i(E_{0}) \subseteq V_{0}$, 
$i(E_{i}) \subseteq V_{i-1}$,  if $1 \leq i \leq 7$,
$ t(E_{i}) \subseteq V_{i}$,if $0 \leq i \leq 6$, and 
$t(E_{7}) \subseteq V_{6}$. That is, $V_{0}$ and $V_{6}$ are two components of
$G^{7}_{1}$ and the remaining edges move from $V_{0}$ to $V_{6}$.

Observe that the permutation group on $2$-symbols acts in an obvious manner
on all these objects.

We consider the free abelian  group on $G_{1}^{6}$ with an inductive limit 
given by $G^{7}_{1}$. Before doing so, we must take a quotient, 
moding out by the subgroup generated by vertices $v$ with $v = v \cdot \alpha$
and all elements of the form $v - v\cdot(1,2)$, 
where $\alpha$ interchanges entries
$0$ and $1$. The first means that we are 
removing $V_{0}$ from consideration. Among the remaining vertices, 
only the ones in $V_{6}$ are the terminus of a path of any length greater than $6$ 
and it follows that the inductive limit only using that part. On  the other hand, 
when we quotient by $v - v \cdot \alpha$,  the result 
is clearly isomorphic to the free abelian group on $H^{6}$, with generating set
$ (\xi^{0}(y), \xi^{1}(y))$, $y $ in $H^{6}$. 
We conclude that 
$C^{s}_{\mathcal{Q}, \mathcal{A}}(\pi_{\xi})_{1,0}$ is isomorphic to
to the inductive limit of this group under the map induced by $H^{7}$.
We have proved the following.

\begin{lemma}
\label{hom:40}
We have 
$C^{s}_{\mathcal{Q}, \mathcal{A}}(\pi_{\xi})_{1,0} \cong D^{s}(H)$.
\end{lemma}

We are now left to consider the boundary map between the two groups and 
a key technical point in its computation is the following. Our notation 
is slightly different from \cite{Pu:HSm}, so  we give a self-contained statement 
here. 

 \begin{lemma}
 \label{hom:50}
 For the regular $s$-bijective factor map 
 $\pi_{\xi}: (X_{G}, \sigma) \rightarrow (X_{\xi}, \sigma_{\xi})$, the 
 constant $K_{\xi} = 0$ satisfies the conditions of Lemma 2.7.2 of \cite{Pu:HSm}.
 That is, if $x^{1}, x^{2}, y^{1}, y^{2}$ are all in 
 $X_{G}$ with
  $\pi_{\xi}(x^{1}) = \pi_{\xi}(x^{2}), \pi_{\xi}(y^{1}) = \pi_{\xi}(y^{2})$, 
  $x^{1}_{k} = y^{1}_{k}$, for $k \geq k_{0}$, and $x^{2}, y^{2}$ stably
  equivalent, then $x^{1}_{k} = y^{1}_{k}$, for $k \geq k_{0}$ also.
 \end{lemma}
 
 \begin{proof}
 The statement is trivial if either $x^{1} = x^{2}$ or $y^{1} = y^{2}$. It remains 
 to consider when both pairs are distinct. This situation is described
 explicitly in Theorem \ref{smale:50}. The rest of the proof is done by checking 
 the different cases, which we leave to the reader.
 \end{proof}
 
 \begin{lemma}
 \label{hom:60}
 The boundary map 
 \[
 d^{s}_{\mathcal{Q}} : C^{s}_{\mathcal{Q}, \mathcal{A}}(\pi_{\xi})_{1,0} 
 \rightarrow 
 C^{s}_{\mathcal{Q}, \mathcal{A}}(\pi_{\xi})_{0,0}
 \]
 is the zero map. 
 \end{lemma}
 
\begin{proof}
Via the isomorphism 
 of Lemma \ref{hom:40}, this group is generated by classes of the form 
 $ (\xi^{0}(y), \xi^{1}(y))$, $y $ in $H^{6}$. Using the formula given 
 in Definition 4.2.1, we have 
 \[
 d^{s}_{\mathcal{Q}}(\xi^{0}(y), \xi^{1}(y)) = \xi^{0}(y) -  \xi^{1}(y), 
 \]
 for each $y$ in $H^{6}$, where the term on the right
 is interpreted as an element of $\Z G^{6}$. We know that the map $t$ induces 
 an isomorphism between $D^{s}(G^{7})$ and $D^{s}(G)$ and we have 
 \[
 t( \xi^{0}(y) -  \xi^{1}(y) ) = t(\xi^{0}(y)) - t(\xi^{1}(y)) = 0, 
 \]
 by condition (H0).
 \end{proof}
 
 The computations of the homology groups, $H^{s}_{*}(X_{\xi}, \sigma_{\xi})$,
  as summarized in Theorem \ref{hom:20}, follow
 from Theorem \ref{hom:30} and Lemmas \ref{hom:40} and \ref{hom:60}.
 
 Let us mention an alternate proof of these results. This is based on 
 two sets of results. The first are those of Proietti and Yamashita 
 \cite{PY:hom1} and \cite{PY:hom2}, who show that the Smale space 
 homology here agrees with the groupoid homology as studied by Matui.
 Further, Matui has recently the adapted results which are used in the next 
 section for the $K$-theory of the $C^{*}$-algebras to the case 
 of groupoid homology: see \cite{Ma:homex}.
  These require the unit space to be totally disconnected
 and so only apply to parts  1, 2 and 5 of Theorem \ref{hom:20}.

\section{Groupoids, $C^{*}$-algebras and $K$-theory}
\label{c*}

The goal of this section is to describe the 
stable and unstable equivalence relations (or groupoids) 
of our Smale space, $(X_{\xi}, d_{\xi}, \sigma_{\xi})$, 
their associated $C^{*}$-algebras and their $K$-theory.

We begin with a short review of what is 
involved for general Smale space.
 For any Smale space, $(X, d, \varphi)$, two points 
$x,y$ in $X$ are \emph{stably} equivalent if
\[
\lim_{n \rightarrow + \infty} d(\varphi^{n}(x), \varphi^{n}(y)) = 0.
\]
\emph{Unstable} equivalence is defined by simply replacing
both occurrences  of  $\varphi$ by $\varphi^{-1}$.
We define 
\begin{eqnarray*}
G^{s}(X, \varphi) & = & \{ (x,y) \in X \times X \mid \lim_{n \rightarrow +\infty}  
d(\varphi^{n}(x), \varphi^{n}(y)) = 0 \},  \\
G^{u}(X, \varphi) & = & \{ (x,y) \in X \times X \mid \lim_{n \rightarrow +\infty}  
d(\varphi^{-n}(x), \varphi^{-n}(y)) = 0 \}.
\end{eqnarray*}

Each is an equivalence relation and has a natural topology. It will not be 
necessary to describe this here in detail, but we refer the reader to 
\cite{Pu:Sma}. In addition, if  $(X, \varphi)$ is non-wandering, 
then, as groupoids, each has a Haar system. Again, we do not need a detailed 
description. For any $x$ in $X$, we let $X^{s}(x)$ and $X^{u}(x)$
 be the equivalence classes of $x$ in the two equivalence
  relations. Each carries a natural topology which is finer than 
the relative topology of $X$ but makes each locally compact Hausdorff.

Generally speaking, \'{e}tale groupoids \cite{Ren:book} are much
 easier to deal with
and, thanks to the work of Muhly, Renault and 
Williams \cite{MRW:eqgd},
 our groupoids are 
equivalent to \'{e}tale groupoids by reducing to an abstract transversal. 
With the slightly stronger hypothesis that $(X, \varphi)$ is 
irreducible, we select a finite
set $P$ with the property that $\varphi(P)=P$ (i.e. consisting of periodic 
points), we let 
\[
X^{u}(P) = \cup_{p \in P} X^{u}(p), \hspace{1cm}
 X^{s}(P) = \cup_{p \in P} X^{s}(p)
\]
and these function as natural transversals to
 $G^{s}(X, \varphi), G^{u}(X, \varphi)$, 
respectively. Note that these spaces are given locally compact topologies
which are finer than the relative topology of $X$.
We refer the reader to \cite{PS:Sma}.
We denote the reduction to these transversals by
\begin{eqnarray*}
G^{s}(X, \varphi, P) & = & G^{s}(X, \varphi) \cap 
\left( X^{u}(P) \times X^{u}(P) \right), \\
G^{u}(X, \varphi, P) & = & G^{u}(X, \varphi) \cap 
\left( X^{s}(P) \times X^{s}(P) \right)
\end{eqnarray*}
Each is an \'{e}tale equivalence relation. Different choices 
of $P$ yield equivalent groupoids. In \cite{PS:Sma},
 it is shown that these groupoids are amenable, so we do 
 not need to make a distinction between 
 the full and reduced $C^{*}$-algebras. We define
 \begin{eqnarray*}
 S(X, \varphi, P) & = & C^{*}(G^{s}(X, \varphi, P) ), \\
 U(X, \varphi, P) & = & C^{*}(G^{u}(X, \varphi, P) ).
 \end{eqnarray*}
 
 The map which sends  a continuous function of compact support, $f$,
  on either
 of the two groupoids, to $f \circ (\varphi \times \varphi)^{-1}$ extends 
 to an automorphism of the respective $C^{*}$-algebra, which is denoted
 by $\varphi$.  The Ruelle 
 $C^{*}$-algebras are defined as the associated crossed-products:
 \begin{eqnarray*}
 R^{s}(X, \varphi, P) & = & S(X, \varphi, P) \rtimes_{\varphi} \Z, \\ 
 R^{u}(X, \varphi, P) & = & U(X, \varphi, P) \rtimes_{\varphi} \Z.
 \end{eqnarray*}
 
 Let us review some basic properties of these $C^{*}$-algebras. If $(X, \varphi)$
 is mixing then $S(X, \varphi,P)$ and $U(X, \varphi, P)$ are 
 separable, simple, finite, 
 stable, satisfy the Universal Coefficient Theorem (UCT) \cite{PS:Sma} and
 have finite nuclear dimension \cite{DS:Sma}. If  $(X, \varphi)$
 is irreducible, then $R^{s}(X, \varphi, P) $ and $R^{u}(X, \varphi, P)$
 are separable, simple, purely infinite, stable, 
 satisfy the Universal Coefficient Theorem (UCT) \cite{PS:Sma} and
 have finite nuclear dimension \cite{DS:Sma}. 
   All of these $C^{*}$-algebras come under the Elliott 
 classification scheme. ( We include  a small remark: the 
 results of \cite{DS:Sma} need the hypothesis that the $C^{*}$-algebras 
 contain a projection. This was later shown to hold in general \cite{DGY:proj}, 
 but, in our case, we will explicitly provide clopen subsets of the 
 unit spaces of both $G^{s}(X_{\xi}, \sigma_{\xi}, P_{\xi})$ and 
 $G^{s}(X_{\xi}, \sigma_{\xi}, P_{\xi})$ which shows this holds.)
 
In our particular case for $(X_{\xi}, d_{\xi}, \sigma_{\xi})$, 
we can actually go a little further and the first step is to make a
 particular choice for $P$, as we describe.

Under the standing hypotheses, we
 select a periodic point as follows. We have already observed that
  $G^{1}-\xi(H^{1})$  is also primitive. 
 We select a cycle $C^{1} \subseteq G^{1}-\xi(H^{1})$ of minimal length.
 This implies that $t|_{ C^{1}}$ is injective. 
 We let $C^{0} = t(C^{1})$. We let $P \subseteq X_{G}$ be the finite 
 set of infinite paths which simply repeat the cycle $C^{1}$.
  We note that an element  $p$ of $P$ is uniquely determined in $P$ by $i(p_{1})$
  in $C^{0}$.
 Also notice that $\pi_{\xi}$ is a bijection from 
 $P$ to $P_{\xi} = \pi_{\xi}(P)$ from part 3 of Theorem \ref{smale:50}.
 
 The $s$-bijective map
 $\pi_{\xi}:X_{G} \rightarrow X_{\xi}$ induces maps at the level 
 of groupoids and also $C^{*}$-algebras
 \begin{eqnarray*}
 (\pi_{\xi})^{*} : & S(X_{\xi}, \sigma_{\xi}, P_{\xi}) & \rightarrow 
 S(X_{G}, \sigma, P),  \\
  (\pi_{\xi})_{*} : &  U(X_{G}, \sigma, P) & \rightarrow 
U(X_{\xi}, \sigma_{\xi}, P_{\xi}).  
 \end{eqnarray*}
 
 The $C^{*}$-algebras $S(X_{G}, \sigma, P)$ and $U(X_{G}, \sigma, P)$ are both 
 AF-algebras, their $K$-zero groups are 
 $D^{u}(G)$ and $D^{s}(G)$, respectively, while their $K$-one groups are 
 trivial.
 
 We now state the two main results of this section. Of course, they
 could easily be assembled into a single result,  but we separate them 
 as the proofs are rather long and also rather different and
 they will be divided into two subsections.

\begin{thm}
\label{c*:40}
If $G, H, \xi$ satisfy the standing hypotheses, 
 we have
\begin{enumerate}
\item 
$K_{0}(S(X_{\xi}, \sigma_{\xi}, P_{\xi})) \cong D^{u}(G)$
as ordered abelian groups
and, under this isomorphism,  the automorphism induced by $\sigma_{\xi}$
is $A_{G}^{T}$.
\item 
$K_{1}(S(X_{\xi}, \sigma_{\xi}, P_{\xi})) \cong D^{u}(H)$
and, under this isomorphism,  the automorphism induced by $\sigma_{\xi}$
is $A_{H}^{T}$.
\item $K_{0}(R^{s}(X_{\xi}, \sigma_{\xi}, P_{\xi})) \cong \Z^{G^{0}}/ (I-A_{G}^{T})\Z^{G^{0}}
\oplus \ker( I - A_{H}^{T}: \Z^{H^{0}}  \rightarrow  \Z^{H^{0}} ).$
\item $K_{1}(R^{s}(X_{\xi}, \sigma_{\xi}, P_{\xi})) \cong \Z^{H^{0}}/ (I-A_{H}^{T})\Z^{H^{0}}
\oplus \ker( I - A_{G}^{T}: \Z^{G^{0}}  \rightarrow  \Z^{G^{0}} ).$
\end{enumerate}
\end{thm}

\begin{thm}
\label{c*:50}
If $G, H, \xi$ satisfy the standing hypotheses, 
 we have
\begin{enumerate}
\item 
$K_{0}(U(X_{\xi}, \sigma_{\xi}, P_{\xi})) \cong D^{s}(G)$
as ordered abelian groups
and, under this isomorphism,  the automorphism induced by $\sigma_{\xi}$
is $A_{G}^{-1}$.
\item 
$K_{1}(U(X_{\xi}, \sigma_{\xi}, P_{\xi})) \cong D^{s}(H)$
and, under this isomorphism,  the automorphism induced by $\sigma_{\xi}$
is $A_{H}^{-1}$.
\item $K_{0}(R^{u}(X_{\xi}, \sigma_{\xi}, P_{\xi})) \cong \Z^{G^{0}}/ (I-A_{G})\Z^{G^{0}}
\oplus \ker( I - A_{H}: \Z^{H^{0}}  \rightarrow  \Z^{H^{0}} ).$
\item $K_{1}(R^{u}(X_{\xi}, \sigma_{\xi}, P_{\xi})) \cong \Z^{H^{0}}/ (I-A_{H})\Z^{H^{0}}
\oplus \ker( I - A_{G}: \Z^{G^{0}}  \rightarrow  \Z^{G^{0}} ).$
\end{enumerate}
\end{thm}

\begin{rem}
\label{c*:60}
Both results are undoubtedly true if the hypothesis that $G$ is 
primitive is weakened to $(X_{G}, \sigma_{G})$ being non-wandering.
However, our proofs rely heavily on results of two papers 
\cite{Pu:preK} and \cite{Ha:fac} which  require the respective
equivalence relations $G^{s}(X_{G}, \sigma, P)$ and 
$G^{u}(X_{G}, \sigma, P)$ to have all equivalence classes dense 
and this is equivalent to $G$ being primitive.
It would not be difficult to adjust these results to the more general 
situation, but it would complicate matters.
\end{rem}

 \begin{rem}
 \label{c*:70}
The alert reader may be somewhat perplexed by the switching of 's' and 'u',
between the two sides of the first two parts of both results.
This is due to an historical anomaly: the superscript on $G^{s}$ was 
chosen for 'stable' equivalence and it seems logical to use $S$ 
for its $C^{*}$-algebra. However, the elements of the $K_{0}$-group
of this $C^{*}$-algebra are realized by characteristic functions  
of clopen subsets of its unit space, which is $X^{u}(P)$. So the notation
for $D^{u}(G)$ was chosen to coincide with this interpretation.
\end{rem}

Let us mention one other groupoid, without going into 
great detail. It follows from Theorems \ref{constr:190}
and \ref{constr:210} that the map $\sigma_{\xi}$ of $X_{\xi}^{+}$ is
locally expanding and also a local homeomorphism. We can 
associate to it its Deaconu-Renault groupoid, as follows.
For each $x,y$ in $X_{\xi}^{+}$ and positive integers $m,n$
satisfying $\sigma_{\xi}^{m}(x) = \sigma_{\xi}^{n}(y)$, we consider the 
triple $(x, m-n, y)$. The set of all such triples is a groupoid with the product
$(x, k, y) \cdot(x', k', y')$ defined when $y=x'$ and result
$(x, k+k', y')$. It can be given a topology it which it is \'{e}tale 
\cite{D:grpd}. 
The associated groupoid $C^{*}$-algebra has the nice feature
of being unital, and it is also equivalent to 
$R^{u}(X_{\xi}, \sigma_{\xi}, P_{\xi})$.

The result of our K-theory computations
has an interesting consequence that 
our construction exhausts all possible Ruelle 
algebras from mixing Smale spaces, as we describe below.

We begin with some simple  observations. 
Let $(X, \varphi, d)$ be a mixing Smale space 
and $Q$ be a finite $\varphi$-invariant subset of $X$.
The $C^{*}$-algebras  
$R^{s}(X, \varphi, Q)$  and $R^{u}(X, \varphi, Q)$ are 
Spanier-Whitehead duals of each other (see Definition 4.1 and 
Theorem 1.1 of \cite{KPW:dual}). As noted in section 4 of \cite{KPW:dual},
this implies that their $K$-groups are finitely generated. 
Proietti and Yamashita   have recently shown that
$K_{*}(S(X, \varphi))$ and $K_{*}(U(X, \varphi))$ are finite rank 
(Theorem 5.1 of \cite{PY:hom2})
and, 
as explained in section 4 of \cite{KPW:dual}, this implies that 
$K_{0}(R^{s}(X, \varphi, Q))$  and $K_{1}(R^{s}(X, \varphi, Q))$ have the same rank.
Putting all of this together, we conclude that there is an integer $k \geq 0$
and finite abelian groups $G_{0}$ and $G_{1}$ such that 
\begin{eqnarray*}
K_{0}(R^{s}(X, \varphi, Q)) & \cong & \Z^{k} \oplus G_{0}, \\
K_{1}(R^{s}(X, \varphi, Q)) & \cong & \Z^{k} \oplus G_{1}.
\end{eqnarray*}
We add that another consequence of Theorem 5.1 of 
\cite{PY:hom2}, as described in section 4 of \cite{KPW:dual} is that 
\[
R^{s}(X, \varphi, Q) \cong R^{u}(X, \varphi, Q).
\]

\begin{cor}
\label{c*:80}
Let $(X, \varphi, d)$ be a mixing Smale space and $Q$ be
 a finite $\varphi$-invariant subset of $X$. There exist finite directed 
 primitive graphs,
$G, H$ and embeddings $\xi^{0}, \xi^{1}: H \rightarrow G$ satisfying 
(H0), (H1) and (H2) such that \newline
$R^{s}(X_{\xi}, \sigma_{\xi}, P_{\xi}) \cong R^{s}(X, \varphi, Q)$,
for any $P$, a finite $\sigma$-invariant subset of $X_{G}$.
\end{cor}

We begin the proof with a simple algebraic result.

\begin{lemma}
\label{c*:85}
Let  $G$ be a  finitely generated 
abelian group and $d_{0}, M_{0}$ be 
positive integers. There exist $d \geq d_{0}$ and a $ d \times d$
 integer matrix $A$
such that 
\[
\Z^{d} / \left( I-A \right) \Z^{d}  \cong   G
\]
and each entry of $A$ is greater than or equal to $M_{0}$.
\end{lemma}

\begin{proof}
From the classification theorem for finitely generated 
abelian groups (Theorem 2.2, page 76
of Hungerford \cite{Hu:book}), 
we know that there are integers $k \geq 0$, 
$j_{1}, \ldots, j_{l} \geq 2$, $l \geq 0$ such that 
\[
G \cong \Z^{k} \oplus \left( \Z/j_{1} \Z \oplus \cdots \oplus 
\Z / j_{l} \Z \right).
\]
Let $D$ be the diagonal matrix whose diagonal entries are $0$, $k$ times, 
$j_{1}, \ldots, j_{l}$ and then a collection of $1$'s so that there is at least one
and $d_{0}-l-k$,  so the the size of $D$, which we call $d$,  is
at least $d_{0}$.

It is immediate that $\Z^{d}/ D \Z^{d} \cong G$. Observe that this isomorphism 
still holds if we multiply $D$ on either side by an integer matrix 
of determinant one. In particular, it remains true if we add a multiple 
of a row or column to another. Using column operations, 
the $1$ in the $d,d$-entry of $D$ can then 
be used to ensure row $d$ has all positive integer entries. Then using 
row operations,
we can ensure every row, other than row $d$ has entries at least $M_{0}$. 
Finally, we add row $1$ to row $d$ and the resulting matrix matrix $B$ 
has all entries at least $M_{0}$. Then  $A = B+I$ satisfies the desired 
conclusion.
\end{proof}

\begin{proof}[Corollary \ref{c*:80}]
We apply Lemma \ref{c*:85} to the group $G = K_{1}(R^{s}(X, \varphi, Q))$ 
with $d_{0}=M_{0} =1$ to obtain a $d \times d$ matrix $A$. We apply
Lemma \ref{c*:85} a second time to the group 
$Tor(K_{0}(R^{s}(X, \varphi, Q)))$, with $d_{0}=d$ and $M_{0}  = \max\{ 2 A(i,j) +1 
\mid 1 \leq i, j \leq d \}$ to obtain a $d' \times d'$ matrix $B$.

We let $H$ be the directed graph with $d$ vertices and 
adjacency matrix $A^{T}$ and $G$ be the directed graph with $d'$ vertices and 
adjacency matrix $B^{T}$. As $d' \geq d$, we may find an embedding of $H^{0}$ in 
$G^{0}$ and from the choice of $B$ we can find 
embeddings $\xi^{0}, \xi^{1}$ 
of $H^{1}$ in $G^{1}$ which satisfy (H0), (H1) and (H2).

We note that as 
\[
\Z^{d'} / \left( I-B \right) \Z^{d'} \cong Tor(K_{0}(R^{s}(X, \varphi, Q))),
\]
which is finite, $\ker \left( I- B \right)$ is free abelian 
and rank zero, so is trivial.
If we let $k $ be the rank of $K_{1}(R^{s}(X, \varphi, Q))$, then 
$\ker \left( I-A \right)$ also has rank $k$ and is free abelian.
We conclude from Theorem \ref{c*:40} that
\begin{eqnarray*}
K_{0}(R^{s}(X, \varphi, Q))  \cong & \Z^{k} \oplus Tor( K_{0}(R^{s}(X, \varphi, Q)) ) 
   \cong & K_{0}(R^{s}(X_{\xi}, \sigma_{\xi}, P_{\xi})) \\
  K_{1}(R^{s}(X, \varphi, Q))   
   \cong  & K_{1}(R^{s}(X_{\xi}, \sigma_{\xi}, P_{\xi})) &
\end{eqnarray*}
As noted in \cite{KPW:dual} these $C^{*}$-algebras fall under the 
Kirchberg-Phillips classification Theorem, so 
 $R^{s}(X, \varphi, Q) \cong R^{s}(X_{\xi}, \sigma_{\xi}, P_{\xi})$.
\end{proof}

\subsection{Stable equivalence}

In this subsection, we focus on the $C^{*}$-algebras 
$S(X_{\xi}, \sigma_{\xi}, P_{\xi})$ and \newline
$R^{s}(X_{\xi}, \sigma_{\xi}, P_{\xi})$.

\begin{lemma}
\label{c*:100}
\begin{enumerate}
\item The set
\[
Y^{u}_{G}(P) = \{ y \in X_{G}^{u}(P) \mid y_{n} = p_{n}, n \leq 0, 
\text{ some } p \in P \}
\]
is a compact open subset of $X_{G}^{u}(P)$. Its relative topology 
from $X_{G}^{u}(P)$ agrees with the  relative topology from $X_{G}$.
If $x$ is any point of $X_{G}$, then $x$ is stably equivalent
to some $y$ in $Y^{u}_{G}(P)$. Finally, 
we have $\sigma^{-1}(Y^{u}_{G}(P)) \subseteq Y^{u}_{G}(P)$.
\item The set
$Y^{u}_{\xi}(P) = \pi_{\xi}(Y^{u}_{G}(P))$
is a compact open subset of $X_{\xi}^{u}(P)$. Its relative topology 
from $X_{\xi}^{u}(P)$ agrees with the  relative topology from $X_{\xi}$.
If $x$ is any point of $X_{\xi}^{u}(P)$, then $x$ is stably equivalent
to some $y$ in $Y^{u}_{\xi}(P)$. Finally, 
we have $\sigma_{\xi}^{-1}(Y^{u}_{\xi}(P)) \subseteq Y^{u}_{\xi}(P)$.
\end{enumerate}
\end{lemma}

\begin{proof}
For the first part, the statements about the topology are standard.
As $G$ is primitive, we may $l > 0$ such that there is a path of length $l$
between any two vertices of $G$. If $x$
in in $X_{G}$, let $p$ be any point of $P$. Define $y$ by 
$y_{n} = p_{n}, n \leq 0$, $y_{1}, \ldots, y_{l}$ is any path from $t(p_{0})$
to $t(x_{l})$, and $y_{n} = x_{n}, n > l$. So $y$ is stably equivalent to $x$ and lies in 
$Y^{u}_{G}(P)$.

For the second part, $Y^{u}_{\xi}(P) $ is compact since $Y^{u}_{G}(P)$ is. 
Moreover,  $Y^{u}_{G}(P)$ is clearly invariant under $\sim_{\xi}$ and 
$\pi_{\xi}$ is continuous, so 
$Y^{u}_{G}(P) = \pi_{\xi}^{-1}(Y^{u}_{\xi}(P)) $ which implies that 
$Y^{u}_{\xi}(P)$ is open also.
The remaining statements are clear.
\end{proof}

The following is an immediate consequence.

\begin{prop}
\label{c*:110}
Let us denote by $G^{s}(\xi, P)$ the reduction of the groupoid 
$G^{s}(X_{\xi}, \sigma_{\xi})$ to $Y^{u}_{\xi}(P)$. It is \'{e}tale,
equivalent to $G^{s}(X_{\xi}, \sigma_{\xi})$ and also \newline 
$G^{s}(X_{\xi}, \sigma_{\xi}, P_{\xi})$, in the sense of Muhly, Renault and Williams
\cite{MRW:eqgd}. In particular, we have containments and a commutative diagram
\[
\xymatrix{
C^{*}( G^{s}(\xi,  P)) 
  \ar^{\sigma_{\xi}^{-1}}[d] & \subseteq &
C^{*}( G^{s}(X_{\xi}, \sigma_{\xi}, P_{\xi})) \ar^{\sigma_{\xi}^{-1}}[d] \\
C^{*}( G^{s}(\xi,  P)) 
  &  \subseteq  & 
C^{*}( G^{s}(X_{\xi}, \sigma_{\xi}, P_{|xi})). }
\] 
\end{prop}

The proof of Theorem \ref{c*:40} involves an application of the results 
 of \cite{Ha:fac}. Let us provide a brief discussion of the set-up there.
 It will be fairly similar to the one we consider here, but
  there are a couple of differences which we need to address.
  
 We begin with  two Bratteli diagrams, $ (V, E), (W, F)$ and 
 two embeddings of the latter into the former satisfying conditions 
 analogous to
 (H0) and (H1). Section 3 of \cite{Ha:fac} describes the construction
 of a quotient of the path space of the Bratteli diagram $X_{E}$, denoted
 $X_{\xi}$. The equivalence relation of tail equivalence on 
 $X_{E}$, denoted, $R_{E}$ then descends to an \'{e}tale equivalence relation, 
 denoted $R_{\xi}$.  It is then shown (Theorem 1.1) that the $K_{0}(C^{*}(R))$ 
 is isomorphic to the dimension group of the Bratteli diagram $(V , E)$, while 
 $K_{1}(C^{*}(R))$ is isomorphic to that of $(W, F)$.
 
 The first point to note is that in our current situation will 
 we be using $(V, E)$ and $(F, W)$ stationary diagrams given by the matrices
 $G$ and $H$, respectively. That is,  at least approximately, 
 $V_{n}=G^{0}, E_{n}=G^{1}, W_{n}=H^{0}, F_{n}=H^{1}$, for all $n$.

 There is a second minor problem with the convention of \cite{Pu:preK} that
  the Bratteli diagram 
 begins with $V_{0}$ as a single vertex. This is also solved easily:
 we let $V_{0}$ be a single vertex, $V_{1} =  t(C)$, our selected minimal
 cycle, and $E_{1}$ to have one edge from $V_{0}$ to each vertex, $v$,
  of $V_{1}$
 which is simply the unique edge $e$ of $C$ with $t(e)=v$. 
 Then, for
 $n \geq 2$, we inductively $E_{n} = i^{-1}(V_{n-1})$ and $V_{n} = t(E_{n})$. 
  As we assume that $G$ is primitive, there will be some 
 $n_{0} \geq 1$ such that $V_{n} = G^{0}$ and $E_{n} = G^{1}$, for all 
 $n \geq n_{0}$. With these definitions, it is immediate that
 the path space $X_{E}$ of \cite{Ha:fac} coincides with $Y^{u}_{G}(P)$, 
 the space $X_{\xi}$ of \cite{Ha:fac} coincides with $Y^{u}_{\xi}(P)$,
 the groupoid $R_{E}$ coincides with the reduction of
  $G^{s}(X_{G}, \sigma, P)$ to $Y^{u}_{G}(P)$ 
 and the groupoid $R_{\xi}$ coincides with $G^{s}(\xi, P)$.
 
  There is a somewhat different solution to this problem for $(W, F)$. 
 We simply drop the assumption on the start and say $W_{n}$ is only defined for 
 $n \geq n_{0}$, the two embeddings $\xi^{0}, \xi^{1}$ are  the given ones.
 The results of \cite{Ha:fac} still hold verbatim.
 
For the third and fourth parts of Theorem \ref{c*:40}, we use the 
Pimsner-Voiculescu exact sequence \cite{RLL:book}. For notational purposes, 
we use $S$ instead of $S(X_{\xi}, \sigma_{\xi}, P_{\xi})$ and
  $R^{s} = S \rtimes_{\sigma_{\xi}} \Z$:

\vspace{.5cm}
  
\hspace{2cm}
\xymatrix{ K_{0}(S) \ar^{id - (\sigma_{\xi})_{*}}[r] & 
K_{0}(S) \ar[r] & K_{0}(R^{s}) \ar[d] \\
 K_{1}(R^{s}) \ar[u] &  K_{1}(S) \ar[l] &
  K_{1}(S) \ar^{id - (\sigma_{\xi})_{*}}[l] }
  
\vspace{.5cm}
  
  This immediately yields two short exact sequences
  \[
  0 \rightarrow K_{*}(S) / 
 \left( id - (\sigma_{\xi})_{*} \right)K_{*}(S)
 \rightarrow  K_{*}(R^{s}) \rightarrow 
 \ker \left( id - (\sigma_{\xi})_{*} \right) 
  \rightarrow 0.
  \]
  
  We know that $K_{0}(S) \cong D^{u}(G)$, which is an inductive limit
  of groups $\Z^{ G^{0}}$. This means that there is a natural map 
  we denote $i_{1}: \Z^{G^{0}} \rightarrow D^{u}(G)$ as the first group
  in the inductive limit. This intertwines multiplication by $A_{G}^{T}$ 
  with the automorphism $A_{G}^{T}$ and it
   is an easy exercise in algebra to show that
  this induces isomorphisms
  \begin{eqnarray*}
  \ker \left( I - A_{G}^{T} : \Z^{G^{0}} \rightarrow \Z^{G^{0}}  \right)
    &  \cong  & \ker \left( id - (\sigma_{\xi})_{*} \right) \\
  \Z^{G^{0}} / \left( I - A_{G}^{T}    \right) \Z^{G^{0}} 
  &  \cong  & D^{u}(G) / 
 \left( id - A_{G}^{T} \right)D^{u}(G).
 \end{eqnarray*}
 
 The quotients in the two exact sequences above are now seen to be subgroups
 of $\Z^{G^{0}}$ and hence are finitely generated and free. It follows 
 that the two sequences both split. This completes the proof of Theorem \ref{c*:40}.

\subsection{Unstable equivalence}

In this subsection, we focus on the $C^{*}$-algebras 
$U(X_{\xi}, \sigma_{\xi}, P_{\xi})$ and  \newline
$R^{u}(X_{\xi}, \sigma_{\xi}, P_{\xi})$.

\begin{lemma}
\label{c*:200}
\begin{enumerate}
\item
 The set
\[
Y^{s}_{G}(P) = \{ y \in X_{G}^{s}(P) \mid y_{n}=p_{n}, n \geq -1, \text{ some }
 p \in P \}
\]
is a compact open subset of $X_{G}^{s}(P)$. Its relative topology 
from $X_{G}^{s}(P)$ agrees with the  relative topology from $X_{G}$.
If $x$ is any point of $X_{G}$, then $x$ is unstably equivalent
to some $y$ in $Y^{s}_{G}(P)$. Finally, 
we have $\sigma(Y^{s}_{G}(P)) \subseteq Y^{s}_{G}(P)$.
\item The set
$Y^{s}_{\xi}(P) = \pi_{\xi}(Y^{s}_{G}(P))$
is a compact open subset of $X_{\xi}^{s}(P_{\xi})$. Its relative topology 
from $X_{\xi}^{s}(P_{\xi})$ agrees with the  relative topology from $X_{\xi}$.
If $x$ is any point of $X_{\xi}^{s}(P_{\xi})$, then $x$ is unstably equivalent
to some $y$ in $Y^{s}_{\xi}(P)$. Finally, 
we have $\sigma_{\xi}(Y^{s}_{\xi}(P)) \subseteq Y^{s}_{\xi}(P)$.
\end{enumerate}
\end{lemma}

\begin{proof}
The first part is all easy topological facts while the second follows as
$\pi_{\xi}$ is a homeomorphism when restricted to $X^{s}_{G}(P)$.
\end{proof}

The next result follows immediately from the last.

\begin{prop}
\label{c*:210}
Let us denote by $G^{u}(\xi, P)$ the reduction of the groupoid 
$G^{u}(X_{\xi}, \sigma_{\xi})$ to $Y^{s}_{\xi}(P)$. It is \'{e}tale
equivalent to $G^{u}(X_{\xi}, \sigma_{\xi})$ and also to \newline
$G^{u}(X_{\xi}, \sigma_{\xi}, P_{\xi})$, in the sense of Muhly, Renault and Williams
\cite{MRW:eqgd}. In particular, we have containments and a commutative diagram
\[
\xymatrix{
C^{*}( G^{u}(\xi,  P)) 
  \ar^{\sigma_{\xi}}[d] & \subseteq &
C^{*}( G^{u}(X_{\xi}, \sigma_{\xi}, P_{\xi})) \ar^{\sigma_{\xi}}[d] \\
C^{*}( G^{u}(\xi,  P)) 
  &  \subseteq  & 
C^{*}( G^{u}(X_{\xi}, \sigma_{\xi}, P_{\xi})). }
\] 
\end{prop}

We begin our analysis with a technical result on the 
structure of $G^{u}(\xi,  P)$.

\begin{lemma}
\label{c*:220}
Let $x, y$ be  in $Y^{s}_{G}(P)$ and assume that 
$\pi_{\xi}(x)$ and $\pi_{\xi}(y)$ are  unstably equivalent 
in $(X_{\xi}, \sigma_{\xi})$. Then either
\begin{enumerate}
\item 
there exists $N < 0$ such that $x_{n} = y_{n}$, for all $n\leq N$, or
\item 
there exists $N < 0, i =0,1$ and $z$ in $X_{H}$ such that  $x_{n} = \xi^{i}(z_{n}), 
y_{n} = \xi^{1-i}(z_{n})$, for all $n \leq N$.
\end{enumerate}
\end{lemma}

\begin{proof}
We recall that in a general Smale space $(X, \varphi, d)$, if points $x, y$ are
unstably equivalent, then there exists $n_{0} < 0$ such that 
$\varphi^{n_{0}}(x)$ and $\varphi^{n_{0}}(y)$ are in the same local 
unstable set, meaning that $d( \varphi^{n_{0}}(x),\varphi^{n_{0}}(y)) < \epsilon_{X}$
and  $[\varphi^{n_{0}}(x),\varphi^{n_{0}}(y)] = \varphi^{n_{0}}(x)$
(see Proposition 2.1.11 of \cite{Pu:HSm}). It follows by 
induction that for all $n \leq n_{0}$, we have
$d( \varphi^{n}(x),\varphi^{n}(y)) < \epsilon_{X} \lambda^{n-n_{0}}$.
In our case, this means that
  \[
 d_{\xi}( \sigma_{\xi}^{n}(\pi_{\xi}(x) ),
  \sigma_{\xi}^{n}(\pi_{\xi}(y))) < 2^{-3 + n - n_{0}},
  \]
    for all $n \leq n_{0}$. From the definition of the metric $d_{\xi}$ on 
 the inverse limit space $X_{\xi}$, we have 
 $d_{\xi}( \sigma_{\xi}^{n}(\pi_{\xi}(x) )^{0},
  \sigma_{\xi}^{n}(\pi_{\xi}(y)^{0}) < 2^{-3+n -n_{0}}$. Recalling the 
  definition of $\pi_{\xi}$ on $X_{G}$, we have 
  $d_{\xi}( \pi_{\xi}(\chi_{G}^{+}(\sigma^{n}(x ))),
    \pi_{\xi}(\chi_{G}^{+}(\sigma^{n}(y ))) ) < 2^{-3+n - n_{0}}$.

 It follows from Theorem \ref{constr:90} that we have 
 \begin{eqnarray} 
 \label{constr:500}
   d_{G_{\xi}}(\tau_{\xi}( \chi_{G}^{+}(\sigma^{n}(x )) ), 
  \tau_{\xi}( \chi_{G}^{+}(\sigma^{n}(y ))))   &   < &  2^{-3-n_{0} +n} \\
    \label{constr:501} 
    d_{\T}( \theta( \chi_{G}^{+}(\sigma^{n}(x )))
  \theta(\chi_{G}^{+}(\sigma^{n}(y ))) )  &  < & 2^{-3-n_{0} +n},
\end{eqnarray}
for $n \leq n_{0}$.
The first of these inequalities immediately implies that 
$\tau_{\xi}(x_{n}) = \tau_{\xi}(y_{n})$, for all $n < n_{0}$.
If $x_{n}$ is not in $\xi(H^{1})$, it follows that $x_{n}=y_{n}$.

Any $x$ in $Y^{s}_{G}(P)$ is right-tail equivalent to a 
point in $P$ so, for any integer $n $, there is a least 
$m \geq n$ such that
$x_{m}$ is not in $\xi(H^{1})$. We note then that
 \[
 \theta( \chi_{G}^{+}(\sigma^{n}(x ))) = 
 \exp \left( 2 \pi i \sum_{j=1}^{m-n-1}  \varepsilon(x_{n+j}) 2^{-j} \right).
 \]
 As a consequence of the sum being finite, if the quantity above equals 
 one, then $\varepsilon(x_{k}) = 0$, for $n+1 \leq k < m$.
 
 For $0 \leq t < 1$, we let $\exp( 2 \pi i t)^{1/2} = \exp(  \pi i  t)$.
Observe that 
\begin{eqnarray*}
\theta(\chi^{+}_{G}(\sigma^{n-1}(x))) & = &
 \exp( \varepsilon(x_{n}) \pi i) 
 \theta( \chi^{+}_{G}(\sigma^{n}(x)) )^{1/2} \\
  &  =  &
 (-1)^{\varepsilon(x_{n})} \theta( \chi^{+}_{G}(\sigma^{n}(x)) )^{1/2}.
 \end{eqnarray*}
 
 Suppose for some $n < n_{0}-1$, we have $x_{n+1}=y_{n+1}$.
 We will show $x_{n}=y_{n}$ also. There are several cases to consider.
First, suppose that $x_{n+1}=y_{n+1}$ is not in  $\xi(H^{1})$.
If $x_{n}$ is also not in  $\xi(H^{1})$, then $x_{n}=y_{n}$. 
If $x_{n}$ is in  $\xi(H^{1})$, then 
$\theta( \chi_{G}^{+}(\sigma^{n-1}(x ))) = \exp(2 \pi i \varepsilon(x_{n}))$ and 
$\theta( \chi_{G}^{+}(\sigma^{n-1}(y))) = \exp(2 \pi i \varepsilon(y_{n}))$.
The second inequality above then implies that 
$\varepsilon(x_{n}) = \varepsilon(y_{n})$ and hence $x_{n}=y_{n}$ also.

 Now, suppose  $x_{n+1}=y_{n+1}$ is in 
 $\xi(H^{1})$. If $\varepsilon(x_{n+1}) =0$, then both
 $\chi^{+}_{G}(\sigma^{n}(x))$ and $\chi^{+}_{G}(\sigma^{n}(y))$ lie 
 in $\{ \exp(2 \pi i t) \mid 0 \leq t < 2^{-1} \}$. 
 If $\varepsilon(x_{n+1}) =1$, then both
 $\chi^{+}_{G}(\sigma^{n}(x))$ and $\chi^{+}_{G}(\sigma^{n}(y))$ lie 
 in $\{ \exp(2 \pi i t) \mid 2^{-1} \leq t < 1 \}$. 
 In either case, we have
 \[
 d_{\T}( \theta(\chi^{+}_{G}(\sigma^{n}(x)))^{1/2},
   \theta(\chi^{+}_{G}(\sigma^{n}(y)))^{1/2} )
   = 2^{-1}d_{\T}( \theta(\chi^{+}_{G}(\sigma^{n}(x))),
   \theta(\chi^{+}_{G}(\sigma^{n}(y))) )
 \]
 and hence
 \begin{eqnarray*}
 2^{-2-n_{0}+n} & >   & d_{\T}( \theta(\chi^{+}_{G}(\sigma^{n-1}(x))),
   \theta(\chi^{+}_{G}(\sigma^{n-1}(y))) )
      \\
  &  =  & d_{\T}( \theta(\chi^{+}_{G}(\sigma^{n}(x)))^{1/2}
 (-1)^{\varepsilon(x_{n})} ,
   \theta(\chi^{+}_{G}(\sigma^{n}(y)))^{1/2} (-1)^{\varepsilon(y_{n})} )  \\
  & \geq & d_{\T}( \theta(\chi^{+}_{G}(\sigma^{n}(x)))^{1/2}
    (-1)^{\varepsilon(x_{n})} ),
   \theta(\chi^{+}_{G}(\sigma^{n}(x)))^{1/2} (-1)^{\varepsilon(y_{n})}) \\
   &  &    - d_{\T}( \theta(\chi^{+}_{G}(\sigma^{n}(x)))^{1/2}
    (-1)^{\varepsilon(y_{n})} ,
   \theta(\chi^{+}_{G}(\sigma^{n}(y)))^{1/2} (-1)^{\varepsilon(y_{n})} )  \\
  &  = & d_{\T}( 
     (-1)^{\varepsilon(x_{n})} , (-1)^{\varepsilon(y_{n})} ) \\
  &  &  - 2^{-1} d_{\T}( \theta(\chi^{+}_{G}(\sigma^{n}(x)))
    ,\theta(\chi^{+}_{G}(\sigma^{n}(y))) )  \\
   &  \geq & d_{\T}((-1)^{\varepsilon(x_{n})},  (-1)^{\varepsilon(y_{n})} )
 - 2^{-4-n_{0}+n}.
 \end{eqnarray*}
 From which it follows that  $\varepsilon(x_{n}) = \varepsilon(y_{n})$ 
 and so $x_{n}=y_{n}$.

We have shown that if $x_{n+1} = y_{n+1}$, for some 
$ n < n_{0}-1$, then $x_{n} = y_{n}$ also. It follows
 by induction that we are in case 1.

It remains to consider the case  $x_{n} \neq y_{n}$, 
which implies $x_{n}$ is  in $\xi(H^{1})$ and 
$\varepsilon(x_{n}) \neq \varepsilon(y_{n})$, 
for all $n < n_{0}$. Without loss of generality, assume 
$\varepsilon(x_{n})=0, \varepsilon(y_{n})=1$, for some $n < n_{0}$. 
We will show that if $\varepsilon(x_{n-1})=1, \varepsilon(y_{n-1})=0$, 
then $x_{n-2}=y_{n-2}$ contradicting our hypothesis. The only remaining 
possibility is case 2, so this will complete the proof.

Under the conditions 
$\varepsilon(x_{n})=0, \varepsilon(y_{n})=1, 
\varepsilon(x_{n-1})=1, \varepsilon(y_{n-1})=0$, 
we have 
$ \theta(\chi^{+}_{G}(\sigma^{n-2}(x))$ lies in 
$\{ \exp(2 \pi i t) \mid 2^{-1} \leq t < 1-2^{-2} \}$ while 
$ \theta(\chi^{+}_{G}(\sigma^{n-2}(y))$ lies in 
$\{ \exp(2 \pi i t) \mid 2^{-2} \leq t < 2^{-1} \}$. It follows that 
 \[
 d_{\T}( \theta(\chi^{+}_{G}(\sigma^{n}(x)))^{1/2},
   \theta(\chi^{+}_{G}(\sigma^{n}(y)))^{1/2} )
   = 2^{-1}d_{\T}( \theta(\chi^{+}_{G}(\sigma^{n}(x))),
   \theta(\chi^{+}_{G}(\sigma^{n}(y))) )
 \]
 and the same calculation as before shows that 
 $\varepsilon(x_{n-2}) = \varepsilon(y_{n-2})$ and hence 
 $x_{n-2}=y_{n-2}$.
\end{proof}

The proof of Theorem \ref{c*:50} involves an application of the results 
 of \cite{Pu:preK}. Let us provide a brief discussion of the set-up there.
 It will be fairly similar to the one we consider here, but
  there are a couple of differences which we need to address.
  
 We begin with  two Bratteli diagrams, $ (V, E), (W, F)$ and 
 two embeddings of the latter into the former satisfying conditions 
 analogous to
 (H0) and (H1). Section 2 of \cite{Pu:preK} describes the construction
 of an \'{e}tale equivalence relation, $R$, on the path space of the diagram
 $(V, E)$, $X_{E}$, which contains tail equivalence, $R_{E}$, 
 as an open subequivalence relation.
 It is then shown that the $K_{0}(C^{*}(R))$ 
 is isomorphic to the dimension group of the Bratteli diagram $(V , E)$, while 
 $K_{1}(C^{*}(R))$ is isomorphic to that of $(W, F)$.
 
 The first point to note is that in our current situation will 
 we be using $(V, E)$ and $( W, F)$ stationary diagrams given by the matrices
 $G$ and $H$, respectively. That is,  at least approximately, 
 $V_{n}=G^{0}, E_{n}=G^{1}, W_{n}=H^{0}, F_{n}=H^{1}$, for all $n$.
 
  The first minor annoyance is that, because we are studying 
 unstable or left-tail equivalence, our graph $G$ is going in the wrong direction. 
 (As we were also looking at  right-tail equivalence in the last subsection, this
 was kind of inevitable.)  
This can be repaired easily by simply looking at the opposite graphs of $G$ and $H$; 
that is, simply reverse the maps $i,t$. This means that the dimension groups
of $(V, E)$ and $(W, F)$ will be isomorphic to
$D^{s}(G)$ and $D^{s}(H)$, respectively.

 There is a second minor problem with the convention of \cite{Pu:preK} that
  the Bratteli diagram 
 begins with $V_{0}$ and $W_{0}$ as a single vertex. This is also solved easily
 exactly as we did for the stable case.
 
The directional reversal does pose some notational problems. The simplest 
solution is the following.  We observe the following: if $y$ is any path in 
 $Y_{G}^{s}(P)$, then letting $\tilde{y}_{n} = y_{-n}$, for $n \geq 1$, defines
 an infinite path in $X_{E}$. In fact, this association is a homeomorphism between
  $Y_{G}^{s}(P)$ and $X_{E}$ which induces an isomorphism between 
  $G^{u}(X_{G}, \sigma, Y^{s}_{G}(P))$ and $R_{E}$.
Furthermore, bearing in mind that $\pi_{\xi}$ is an homeomorphism between
 $Y_{G}^{s}(P)$ and its image in $X_{\xi}$, 
Lemma \ref{c*:220} shows that $\pi_{\xi}$ induces a bijection between
 $R$, as described in
 \cite{Pu:preK}, and 
$G^{u}(\xi, P)$. We will suppress this map  and simply write 
our sequences as indexed by $n \leq 1$.

One technical issue remains: as a consequence of Lemma \ref{c*:220},
we know that, under this identification, 
$G^{u}(\xi, P)$ and $R$ agree as sets, but the former is given a
 topology based on the Smale space structure, while the latter was given 
 a rather ad-hoc topology in \cite{Pu:preK}. We must check these coincide.
 This 
is the content of the following.

\begin{lemma}
\label{c*:230}
Let $(x,y)$ be in $R$, as above. There exists a positive integer $n$ and
a compact, open neighbourhood $U$ of 
$(x,y)$ in $R$ such that 
\[
\pi_{\xi} \times \pi_{\xi}(U) = 
\{ (z, \sigma^{n}[\sigma^{-n}(y), \sigma^{-n}(z)]) 
\mid z \in \pi_{\xi}(r(U)) \}
\]
is a compact, open subset of $G^{u}(\xi, P)$
and $\pi_{\xi}(r(U))$ is a compact, open subset of $Y^{s}_{\xi}(P)$.
\end{lemma}

\begin{proof}
We first consider the case that $(x,y)$ is in $R_{E}$, meaning that $x,y$
are themselves unstably equivalent in $X_{G}$. We can then choose 
$n \geq 3$ such that $x_{i} = y_{i}$, for all $i \leq -n +3$. We define
$U$ to be the set of all pairs $(z,z')$ such that $z_{i} = z'_{i}$, for $i \leq -n$,
$z_{i} = x_{i}, z'_{i} = y_{i}$, for all $ i \geq -n$. Notice that 
$d_{G}(\sigma^{-n}(y), \sigma^{-n}(z)) \leq 2^{-3}$ and 
$[\sigma^{-n}(y), \sigma^{-n}(z)] = \sigma^{-n}(z')$,
 for all
such $z, z'$. The first desired property of $U$ follow at once. The 
last statements follow from the fact that $\pi_{\xi}$ is a homeomorphism 
on $Y^{s}_{G}(P)$.

The second case, is that $(x,y)$ is not in $R_{E}$. By Lemma 
\ref{c*:220}, we may find $n \geq 1, j=0,1$ and a path $w$ in 
$X_{H}$ such that such that $x_{i} = \xi^{j}(w_{i}), y_{i} = \xi^{1-j}(w_{i})$, 
for all $i \leq -n+3$.
As before, let $U$ be the collection of all pairs $z, z'$ such that 
$z_{i} = x_{i}, z'_{i} = y_{i}$, for all $ i \geq -n$. If we let 
$p = x_{-n+1}, \ldots , x_{ 0}, q = y_{-n+1}, \ldots , y_{ 0}$, then the set 
$\delta_{p,q}^{j, 1-j}$ as defined following Lemma 2.4 in \cite{Pu:preK},
is precisely $U$ and hence is a compact, open subset of $R$.
If $z, z'$ is in $U$, let $z''$ be the unique element   of $X_{G}$ with 
$z'' \neq z, z'' \cong_{\xi} z$. 
It follows that $d_{G}(\sigma^{-n}(y), \sigma^{-n}(z)) \leq 2^{-3}$ and 
we have 
\begin{eqnarray*}
\pi_{\xi}[\sigma^{-n}(y), \sigma^{-n}(z'')] & = & 
   [\pi_{\xi}\sigma^{-n}(y), \pi_{\xi}\sigma^{-n}(z'')] \\
 & = &   [\sigma^{-n}\pi_{\xi}(y), \sigma^{-n}\pi_{\xi}(z'')] \\
  & = &   [\sigma^{-n}\pi_{\xi}(y), \sigma^{-n}\pi_{\xi}(z)]. 
    \end{eqnarray*}
    The first part of the conclusion follows. The last properties are as before.
\end{proof}
 
The computation of the groups $K_{*}(U(X_{\xi}, \sigma_{\xi}, P_{\xi}))$ as stated
in Theorem \ref{c*:50} is an immediate consequence of Theorem 1.1 
of \cite{Pu:preK} and the
construction given there of $R$. 

We next turn to the claim of the maps on these groups induced by $\sigma_{\xi}$. 
First, because of the commutative diagram of Proposition 
\ref{c*:210}, it suffices to compute the map induced by $\sigma_{\xi}$, which 
is   an endomorphism of $C^{*}(G^{u}(\xi, P))$. For the $K_{0}$ group, 
the results of \cite{Pu:preK} actually show that the inclusion of 
$C^{*}(R_{E})$ in $C^{*}(R)$ induces an isomorphism 
on $K_{0}$. The former is an AF-algebra with stationary Bratteli diagram, so the induced
map is given by the matrix $(A_{G}^{T})^{-1}$, as claimed.

For $K_{1}$, let $x$ be an element of $Y^{s}_{G}(P)$ such that, for some $n \leq -1$, 
$i(x_{n})$ is in $\xi^{0}(H^{0})$. Letting $p = (x_{n}, \ldots x_{0})$, Remark 
3.6 of \cite{Pu:preK} gives an explicit description of a partial isometry, 
$v_{p}$, in 
$C^{*}(R)$ so that $v_{p} + (1 - v_{p}^{*}v_{p})$ is a unitary. This gives
an explicit group isomorphism between $D^{s}(G)$ and $K_{1}(C^{*}(R))$. 

Moreover, it is an easy 
exercise to check that if we let $q= (x_{n-1}, \ldots, x_{-1})$, then
 $v_{q} = \sigma_{\xi}(v_{p})$ and it follows that the isomorphism
intertwines the automorphism $(A_{G}^{T})^{-1}$ and 
$(\sigma_{\xi})_{*}$.

The proofs of the last two parts of Theorem \ref{c*:50} are completely analogous 
to those of Theorem \ref{c*:40} and we omit the details.

\section{Realizations}
\label{real}

Our goal in this section is to provide a more concrete
description of the space $X_{\xi}^{+}$. More specifically, we give an explicit 
embedding of it into $\R^{3}$. We also examine a couple of simple examples
more closely.

\begin{defn}
\label{real:10}
For $k \geq 0$, define $\zeta_{k}: X_{k}^{+} \rightarrow \C$ inductively by
setting 
 $\zeta_{0} = \theta$ on $X_{0}^{+}$
and
\[
\zeta_{k}(x) = \left( 1 - 2^{1-n(x)} \right) \theta(x) + 
2^{-3-n(x)}\zeta_{k-1}(\sigma^{n(x)}(x)),
\]
for $x$ in $X_{k}^{+}$ and $k \geq 1$.
\end{defn}

A remark is probably in order on the factor $1 - 2^{1-n(x)}$: this 
is normalized so that, if $n(x)=1$, then the first term is zero.
Ultimately, when we extend the definition of $\zeta_{k}$ to a map
$\zeta$ on 
all of $X_{\xi}^{+}$, this will have the effect that $\zeta(x) =0$ 
exactly when $x_{n}$ is not in $\xi(H^{1})$, for any $n \geq 1$.

\begin{lemma}
\label{real:20}
\begin{enumerate}
\item 
For $x$ in $X_{k}^{+}, k \geq 0$,
 we have
 $\vert \zeta_{k}(x) \vert \leq 1$.
\item 
For any $x,y$ in $X_{k}^{+}, k \geq 0$, we have 
$\vert \zeta_{k}(x) - \zeta_{k}(y) \vert \leq 8 d_{k}(x,y)$,
\item There is a unique continuous map  $\zeta_{\xi}: X_{G}^{+} \rightarrow \C$
such that 
\begin{enumerate}
\item 
for $x$ in $X_{k}^{+}, k \geq 0$, $\zeta_{\xi}(x) = \zeta_{k}(x)$, 
\item 
for $x, y$ in $X_{G}^{+}$, 
 \[
\vert \zeta_{\xi}(x) - \zeta_{\xi}(y) \vert \leq 8 d_{\xi}(x,y).
\]
\item for $x$ in $X_{G}^{+}$, 
$\vert \zeta_{\xi}(x) \vert \leq 1$. 
\end{enumerate}

\end{enumerate}
\end{lemma}

\begin{proof}
We prove the first statement by induction. It
obviously holds for $k=0$.
Let $k \geq 1$, assume the statement is true for $k-1$ and $x$ be in $X_{k}^{+}$. 
We have 
\begin{eqnarray*}
\vert \zeta_{k}(x) \vert & \leq &  1- 2^{1-n(x)} +
 2^{-3-n(x)}\vert \zeta_{k-1}(\sigma^{n(x)}(x)) \vert \\
   & \leq  &  1- 2^{1-n(x)} +
 2^{-3-n(x)}  \\
   & <  & 1.
\end{eqnarray*}

We prove the second part by induction on $k$. For $k = 0$, we have
\[
2 \pi d_{0}(x,y) = 2 \pi d_{\T}(\theta(x), \theta(y)) \geq \vert \theta(x) - \theta(y) \vert 
= \vert \zeta_{0}(x) - \zeta_{0}(y) \vert,
\] 
for all $x, y$ in $X_{0}^{+}$.

We now assume the result holds for $k-1$, with $k \geq 1$. For 
$x, y$ in $X_{k}^{+}$, we consider two 
cases separately. The first is that $(n(x), \theta(x)) \neq (n(y), \theta(y))$.
In this case, we have 
have 
\[
d_{k}(x,y) = d_{G_{\xi}}(\tau_{\xi}(x), \tau_{\xi}(y)) 
+ \vert 2^{-n(x)} - 2^{-n(y)} \vert + d_{\T}(\theta(x), \theta(y)).
\]
We claim that this is bounded below by $2^{-1-n}$, where 
$n = \min \{ n(x), n(y) \}$. Without loss of generality, we assume $n(x) \leq n(y)$.
First consider the case $n(x) < n(y)$, where we have 
\[
d_{k}(x,y) \geq \vert 2^{-n(x)} - 2^{-n(y)} \vert \geq 2^{-n} - 2^{-n-1} = 2^{-1-n}.
\]
We are left to consider $n(x) = n(y)$ and $\theta(x) \neq \theta(y)$. Here, we
have 
\[
d_{k}(x,y) \geq d_{\T}(\theta(x), \theta(y)) \geq 2^{-n},
\]
since  $\theta(x)$ and $\theta(y)$ are distinct $2^{n-1}$-th roots of unity.

On the other hand, we have
\begin{eqnarray*}
\vert \zeta_{k}(x) - \zeta_{k}(y) \vert & \leq  & 
\vert (1 - 2^{1-n(x)})\theta(x) -
(1 - 2^{1-n(y)})\theta(y) \vert  \\
  &   &   + \vert 2^{-2-n(x)} \zeta_{k-1}(\sigma^{n(x)} )
 - 2^{-2-n(x)} \zeta_{k-1}(\sigma^{n(y)} ) \vert. 
 \end{eqnarray*}

 For the first term, we use the triangle inequality as follows
 \begin{eqnarray*}
   &  \vert (1 - 2^{1-n(x)})\theta(x) -
(1 - 2^{1-n(y)})\theta(y) \vert & \\
  \leq & 
\vert (1 - 2^{1-n(x)})\theta(x) -
(1 - 2^{1-n(y)})\theta(x) \vert  & \\
  &    + \vert (1 - 2^{1-n(y)})\theta(x) -
(1 - 2^{1-n(y)})\theta(y) \vert &  \\
   \leq & \vert 2^{1-n(x)} - 2^{1-n(y)} \vert & \\
  &    + (1 - 2^{1-n(y)}) \vert \theta(x) - \theta(y) \vert & \\
    \leq & 2 \vert 2^{-n(x)} - 2^{-n(y)} \vert 
      +  \vert \theta(x) - \theta(y) \vert  & \\
  \leq & 2  \vert 2^{-n(x)} - 2^{-n(y)} \vert 
  + 2 \pi d_{\T}(\theta(x), \theta(y))  & \\
   \leq & 2 \pi  d_{k}(x,y). &
 \end{eqnarray*}

 For the second term, as
  $\zeta_{k-1}$ is in the unit disc, this is bounded 
 by 
 \[
 2^{-2-n(x)}  + 2^{-2-n(y)} \leq  2^{-1-n} \leq d_{k}(x,y)
 \]
 from our claim above. The conclusion follows as
 $2\pi + 1 \leq 8$.

It remains for us to consider the case $(n(x), \theta(x)) = (n(y), \theta(y))$.
Here, we make use of the induction hypothesis in estimating
\begin{eqnarray*}
 \vert \zeta_{k}(x) - \zeta_{k}(y) \vert & = &
  \vert 2^{-3-n(x)}\zeta_{k-1}(\sigma^{n(x)}(x)) - 
  2^{-3-n(y)}\zeta_{k-1}(\sigma^{n(y)}(x)) \vert \\
   & = & 2^{-3-n(x)} \vert \zeta_{k-1}(\sigma^{n(x)}(x)) - 
  \zeta_{k-1}(\sigma^{n(y)}(x)) \vert \\
  & \leq & 2^{-3-n(x)} 8 d_{k-1}( \sigma^{n(x)}(x)), \sigma^{n(y)}(y)) ) \\
  & = & 2^{-n(x)} \left[ d_{G_{\xi}}(\sigma^{n(x)}(x)), \sigma^{n(y)}(y)) \right. \\
     &  & 
    + \left. \lambda_{k-1}( \sigma^{n(x)}(x)), \sigma^{n(y)}(y)) \right]
    \\
    & = &  \cdot 2^{-n(x)} d_{G_{\xi}}(\sigma^{n(x)}(x)), \sigma^{n(y)}(y)) 
     +   2^{2} \lambda_{k}(x,y) \\
     & \leq & \cdot 
     d_{G_{\xi}}(x, y) 
     + 2^{2} \lambda_{k}(x,y) \\
       &  \leq & 8 d_{k}(x,y).
    \end{eqnarray*}

The third part is an immediate consequence of the first two and the definition
of $d_{\xi}$.
\end{proof}

\begin{lemma}
\label{real:35}
Let  $x, y$ be in $X_{G}^{+}$ with $\kappa(x) > 0$. 
\begin{enumerate} 
\item 
We have 
$  (1 - 2^{1-n(x)})\theta(x) - \zeta_{\xi}(x) =
 2^{-3 -n(x) } \zeta_{\xi}(\sigma^{n(x)}(x))$.
\item
If $\zeta_{\xi}(x) = \zeta_{\xi}(y)$  and 
$\tau_{\xi}(x_{m}) = \tau_{\xi}(y_{m})$, for all $1 \leq m \leq n(x)$,
 then $x_{m} = y_{m}$, for all
$ 1 \leq m \leq n(x)$.
\end{enumerate}
\end{lemma}

\begin{proof}
Choose sequences $x^{l}, y^{l}$ in $X_{l}^{+}$, $l \geq 1$ 
converging to $x$ and $y$ respectively. For $l$ sufficiently large, we have
 $x^{l}_{m} = x_{m}, 1 \leq m \leq n(x)$ and  
 $n(x^{l}) = n(x)$, $\theta(x^{l}) = \theta(x)$,  $\sigma^{n(x^{l})}(x^{l})$ 
 is in $X_{l-1}^{+}$ and converges to 
 $ \sigma^{n(x)}(x)$. 
It follows that 
 \begin{eqnarray*}
 \zeta_{\xi}(\sigma^{n(x)}(x) )
 & = & \lim_{l \rightarrow \infty}  \zeta_{l-1}(\sigma^{n(x^{l})}(x^{l}) ) \\
  &  =  & \lim_{l \rightarrow \infty}
 2^{3+n(x^{l})} \left(  (1 - 2^{1-n(x^{l})})\theta(x^{l}) - \zeta_{l}(x^{l}) \right) \\
   &  =  &  2^{3+n(x^{l})} 
   \left(  (1 - 2^{1-n(x)})\theta(x) - \zeta_{\xi}(x) \right)
     \end{eqnarray*}
     which proves the first part.
     
     For the second part, the second
      hypothesis implies that $n(y)=n(x)=n$. Moreover, 
     if $n=1$, then $x_{1}, y_{1}$ are not in $\xi(H^{1})$ and  we are done.
    It remains to consider the case $n > 1$.
     
     We will first show that   
 $\sin( \pi 2^{-n} ) -  2^{-2-n} > 0$ for $n \geq 1$, 
  by induction. For $n=1$, we have 
\[
\sin( \pi 2^{-n} ) -  2^{-2-n} = 1 - 2^{-2} > 0.
\]
Now assume this is positive for some $n \geq 1$. Making use of the  formula
$ \sin(t) = 2 \sin(t/2) \cos(t/2)$, we have
\begin{eqnarray*}
\sin( \pi 2^{-(n+1)} ) -  2^{-2-(n+1)} & = & 
\sin( \pi 2^{-n-1} ) -  2^{-3-n} \\
  &  =  & 2^{-1} \left( \cos(\pi 2^{-n-1}) \right)^{-1} \sin( \pi 2^{-n} )
    - 2^{-3-n} \\
    &  \geq & 2^{-1} \left(  \sin( \pi 2^{-n} ) -  2^{-2-n} \right) \\
    &  >  &  0
    \end{eqnarray*}
    by the induction hypothesis.

  We may choose $l_{0}$ 
 sufficiently large so that, for $l \geq l_{0}$, we have 
 \begin{eqnarray*}
 \vert \zeta_{\xi}(x) - \zeta_{l}(x^{l}) \vert & < & 
  2^{-1} \left( \sin ( \pi 2^{-n}) - 2^{-2-n} \right) ,  \\
 \vert \zeta_{\xi}(y) - \zeta_{l}(y^{l}) \vert & < & 
  2^{-1} \left( \sin ( \pi 2^{-n}) - 2^{-2-n} \right), \\
 x^{l}_{m} & = & x_{m}, \\
 y^{l}_{m} & = & y_{m},
   \end{eqnarray*}
   for all $ 1 \leq m \leq n$. It follows that, for such $l$, 
   \begin{eqnarray*}
   (1- 2^{1-n})\vert \theta(x) - \theta(y) \vert & = & 
  (1- 2^{1-n})\vert \theta(x^{l}) - \theta(y^{l}) \vert \\
   &  \leq & \vert \zeta_{l}(x^{l}) - \zeta_{l}(y^{l}) \vert + 2^{-2-n} \\
    &  \leq & \vert \zeta_{l}(x^{l}) - \zeta_{\xi}(x) \vert  \\
     &  & +  \vert \zeta_{\xi}(y) - \zeta_{l}(y^{l}) \vert  + 2^{-2-n} \\
     & < &  \sin ( \pi 2^{-n} ).
     \end{eqnarray*}
 So we have 
 \[
 \vert \theta(x) - \theta(y) \vert < (1 - 2^{1-n})^{-1} \sin( \pi 2^{-n} ) 
 \leq 2 \sin( \pi 2^{-n}).
 \]
 Both $\theta(x)$ and $\theta(y)$ are among the $2^{n-1}$-th roots 
 of unity and it is an elementary exercise that the minimum distance between 
 any two is $2 \sin( \pi 2^{-n})$. Hence, we see $\theta(x) = \theta(y)$, 
 which implies that $\varepsilon(x_{m}) = \varepsilon(y_{m})$, for all 
 $1 \leq m < n$. Together with the fact that $\tau_{\xi}(x_{m}) = \tau_{\xi}(y_{m})$, 
 for $ 1\leq m \leq n$ implies $x_{m} = y_{m}$, for $ 1 \leq m \leq n$, as claimed.
\end{proof}

 \begin{thm}
 \label{real:40}
Under the standing hypotheses, the map
 from $X_{\xi}^{+}$ to $X_{G_{\xi}}^{+} \times \C$
 sending $\pi_{\xi}(x), x \in X_{G}^{+},$ to 
 the pair $(\tau_{\xi}(x), \zeta_{\xi}(x))$
   is well-defined, continuous and injective.
 \end{thm}
 
 \begin{proof}
 The facts that the  map is well-defined and continuous
  follows immediately from Theorem
  \ref{constr:70} and part 2
 of Lemma \ref{real:20}. It remains for us to prove that
 it is injective. It suffices to show that, for any $x,y$ in $X_{G}^{+}$, 
 if $\tau_{\xi}(x) = \tau_{\xi}(y)$ and $\zeta_{\xi}(x) = \zeta_{\xi}(y)$, then 
 $x \sim_{\xi} y$. The first hypothesis implies that $\kappa(x) = \kappa(y)$. 
 
 If $\kappa(x) = \kappa(y)=0$, then $x, y$ are in $X_{0}^{+}$, so 
 by definition,
 $\zeta_{\xi}(x) = \theta(x), \zeta_{\xi}(y) = \theta(y)$ and hence
 $d_{0}(x,y) = 0$ so $x \sim_{\xi} y$.
 
 We now assume $\kappa(x) = \kappa(y) > 0$. It is an immediate consequence 
 of part 2 of Lemma \ref{real:35} that $x_{m} = y_{m}, 1 \leq m \leq n(x)=n(y)$.
 It is then obvious that 
 $\tau_{\xi}(\sigma^{n(x)}(x)) = \sigma^{n(x)}(\tau_{\xi}(x)) =
 \sigma^{n(x)}(\tau_{\xi}(y))  = \tau_{\xi}(\sigma^{n(x)}(y))$.
 We also note that $\kappa( \sigma^{n(x)}(x)) = \kappa(x) - 1
  = \kappa(y) - 1 = \kappa( \sigma^{n(y)}(y))$.  
 In addition, $x_{m} = y_{m}, 1 \leq m \leq n(x)=n(y)$ implies that 
 $\theta(x) = \theta(y)$ and hence, from part 1 of Lemma \ref{real:35}
 that 
$ \zeta_{\xi}(\sigma^{n(x)}(x)) = \zeta_{\xi}(\sigma^{n(y)}(y))$. 
If $\kappa(x)$ is finite, we may repeat this argument to find a finite $n$ such 
that $x_{m}=y_{m}$ for $ 1 \leq m \leq n$ and 
$\kappa(\sigma^{n}(x)) = \kappa(\sigma^{n}(y))=0$. It follows that 
$x \sim_{\xi} y$. If $\kappa(x)$ is infinite, we may repeat this argument 
to show that 
that $x_{m}=y_{m}$ for all $ 1 \leq m$, so $x=y$.
 \end{proof}

Under our hypotheses, $X_{G_{\xi}}$ is homeomorphic to the 
Cantor ternary set, so we conclude the following.

 \begin{cor}
 \label{real:50}
 Under the standing hypotheses,
 the space $X_{\xi}^{+}$
  can be embedded 
 in $\R \times \C$.
 \end{cor}

  We also observe the following, which already appeared in \cite{Ha:fac}.
   We provide a 
  different proof
  here which employs our metric.

  \begin{cor}
  \label{real:70}
  Recall that we have factor maps
  
  \vspace{.5cm}
 \hspace{2cm} \xymatrix{
  (X_{G}^{+}, \sigma_{G}) \ar^{\pi_{\xi}} [r] & 
  (X_{\xi}^{+}, \sigma_{\xi}) \ar^{\rho_{\xi}} [r]  &
   (X_{G_{\xi}}^{+}, \sigma_{G_{\xi}}) }
   
    \vspace{.5cm} 
\noindent and $\rho_{\xi} \circ \pi_{\xi} = \tau_{\xi}$.   
   
  Under the standing hypotheses, 
  for each $x$ in $X_{G}^{+}$, $\rho_{\xi}^{-1} \{ \tau_{\xi}(x) \}$ is
   \begin{enumerate} 
   \item 
   a finite collection of pairwise disjoint 
    circles if $\kappa(x)$ is finite,
    \item 
 $2^{m}$ points if $ m = \#\{ n \mid x_{n} \in \xi(H^{1}) \}$ is finite,
   \item    
  totally disconnected 
     if $\kappa(x)$ is infinite.
 \end{enumerate}
  \end{cor}

 \begin{proof}
If we apply the map of Theorem \ref{real:40} from $X_{\xi}^{+}$ to 
 $X_{G_{\xi}} \times \C$ and restrict it to 
 $\rho_{\xi}^{-1}\{ \tau_{\xi}(x) \}$, it 
 is a homeomorphism to its image, which is
  $ \{ \tau_{\xi}(x) \} \times \zeta_{\xi}( \tau_{\xi}^{-1}\{ \tau_{\xi}(x) \})$.
  Hence, it suffices for us to prove
   $  \zeta_{\xi}( \tau_{\xi}^{-1}\{ \tau_{\xi}(x) \})$ is as described, 
   in each case. 

For $x$ in $X_{G}^{+}$, define $I(x)$ to 
be the set of  positive integers such that $x_{n}$ is
 not in $\xi(H^{1})$ while $t(x_{n})$ is in $\xi^{0}(H^{0})$. 
For any $n$ in $I(x)$, we define $P(x,n)$ to be the set of all 
 paths $p = (p_{1}, \ldots, p_{n})$ in $G$ such that
  $\tau_{\xi}(p_{m}) = \tau_{\xi}(x_{m})$, for all $ 1 \leq m \leq n$, 
  which is obviously a finite set.
  For such a path  $p$, let $C(p,x)$ denote the set of $y$ in $X_{G}^{+}$ such that
  $y_{m} = p_{m}$, for $1 \leq m \leq n$ and 
  $\tau_{\xi}(y_{m}) = \tau_{\xi}(x_{m})$, for $m > n$.
  
   If $\kappa(x) =0$, then it is a simple matter
  to check that
  \[
  \zeta_{0}( \tau_{\xi}^{-1}( \{ \tau_{\xi}(x) \} ) = \T.
  \]
  If $0 < \kappa(x)< \infty$, then $I(x)$ is finite.
  Let $n$ be its maximum element. It is clear that
  $\tau_{\xi}^{-1} \{ \tau_{\xi}(x) \} = \cup_{p \in P(x,n)} C(p,x)$.
  Again, an easy computation shows that, for each $p$, 
   $\zeta_{\xi}(C(p)) $ is a circle
  (of radius $2^{-3-n} $) and several applications of 
  part 2 of Lemma \ref{real:35} proves 
  that they are pairwise disjoint, for different values of 
  $p$. This proves the first statement.
  
  In the second case, we have $\tau_{\xi}^{-1} \{ \tau_{\xi}(x) \}$ 
  is finite, so its image under $\zeta_{\xi}$ is also.
  
  For the third, it suffices to consider
  the case when $x_{n}$ is in $\xi(H^{1})$, for infinitely many $n$
  and not in $\xi(H^{1})$ for infinitely many $n$. In this case, 
  $I(x)$ is infinite.
  
  Suppose $\tau_{\xi}(y) = \tau_{\xi}(z) = \tau_{\xi}(x)$ and 
  $\zeta_{\xi}(y) \neq \zeta_{\xi}(z)$. This implies that $y \neq z$.
  It follows that we may find $n$ in $I(x)$ and $ 1 \leq m \leq n$ such 
  that $y_{m} \neq z_{m}$. As before, the collection $\zeta_{\xi}(C(p,x)), 
  p \in P(x,n)$, is a finite number of pairwise disjoint circles. Moreover, 
  from part 1 of Lemma \ref{real:35}, 
  $\zeta_{\xi}(\tau_{\xi}^{-1} \{ \tau_{\xi}(x) \})$
  is contained in these circles together with their interior discs.
  These form a partition of   $\zeta_{\xi}(\tau_{\xi}^{-1} \{ \tau_{\xi}(x) \})$
  into pairwise disjoint closed, and hence also open, sets. Moreover, from 
  our choice on $n$, $\zeta_{\xi}(y)$ and $ \zeta_{\xi}(z)$ lie in distinct
  elements. This completes the proof.  
 \end{proof}
 
 \begin{cor}
 \label{real:80}
Under the standing hypotheses, the connected subsets of $X_{\xi}^{+}$ are 
either points or circles and both occur.
 \end{cor}
 
 \begin{proof}
 If $C$ is a connected subset of $X_{\xi}$ then $\rho_{\xi}(C)$ 
 is a connected subset of $X_{G_{\xi}}$ and hence is a single point, say $x$.
 So $C$ is a subset of $\rho_{\xi}^{-1}\{ x \}$ and the 
 conclusion follows
 from Corollary \ref{real:70}. The fact that both circles and single 
 points occur follows from the fact easy fact that both cases 1 and 2 occur
 in \ref{real:70}.
 \end{proof}

 We finish this section by looking at a couple of specific examples.
 The first is instructive even if it does not 
 satisfy hypothesis (H2).
 
 \begin{ex}
 \label{real:100}
 Suppose $G^{0} = H^{0}$ contains a single vertex, $H^{1}$ contains a single edge
 and $G^{1}$ contains exactly two edges. There is essentially only one choice for
 $\xi$.
 The space $X_{G}^{+}$ may be identified
  with $\{ 0, 1\}^{\N}$ in an obvious way. 
 Then $X_{\xi}$ is the unit circle and $\pi_{\xi}$ is binary expansion.
 
  Rather more generally (as noted in \cite{Ha:fac}), if $G^{0}=H^{0}$ and 
  $G^{1} = H^{1} \times \{ 0,1 \}$ with $\xi^{i}$ being the identity map
  on $H^{0}$ and  $\xi^{i}(x) = (x,i)$, for $x$ in $H^{1}$, then 
  $X_{\xi}^{+} $ is homeomorphic to $X_{H}^{+} \times \T$. 
  This can be seen from the first part of Theorem  
  \ref{constr:50} and the fact that
  $X_{k}^{+}$ is empty, for $k \geq 1$.
 \end{ex}

 \begin{ex}
 \label{real:110}
 Suppose $G^{0} = H^{0}$ contains a single vertex, $H^{1}$ contains a single edge
 and $G^{1}$ contains exactly three edges. There is essentially only one choice for
 $\xi$. The space $X_{G}^{+}$ may be identified with $\{ 0, 1, 2\}^{\N}$ in an obvious way.
 Consider the set $A= \{ x \in X_{1}^{+} \mid n(x) \leq 4 \}$.
 The following is a picture of $\zeta_{0}(X_{0}^{+}) \cup \zeta_{1}(A)$:
 
\vspace{1cm}

\hspace{2cm}
\begin{tikzpicture}
  \draw (4,4) circle (4);
  \draw (4,4) circle (1);
   \draw (6,4) circle (.5);
    \draw (2,4) circle (.5);
     \draw (7,4) circle (.25);
    \draw (1,4) circle (.25);
     \draw (4,7) circle (.25);
    \draw (4,1) circle (.25);
     \draw (7.5,4) circle (.12);
    \draw (.5,4) circle (.12);
     \draw (4,7.5) circle (.12);
    \draw (4,.5) circle (.12);
          \draw (6.47,6.475) circle (.12);
                \draw (1.52,6.47) circle (.12);
                      \draw (6.47,1.52) circle (.12);
                            \draw (1.52,1.52) circle (.12);
  \node[below] {Figure 1};                          
    \end{tikzpicture}
  \vspace{1cm}

Hopefully, the reader can see how to sketch all of $\zeta_{1}(X_{1}^{+})$. 
    To get an idea of $\zeta_{2}(X_{2}^{+})$, we suggest the reader verifies 
    the following easy result:
    \[
    \bigcup_{j=0}^{k} \zeta_{\xi}(X_{j}^{+}) = \bigcup 
    \hspace{.5cm} c + r \zeta_{\xi}(X_{0}^{+} \cup X_{1}^{+}), 
    \]
    where the union is over all $c$ in $\C$ and positive real numbers $r$ such that
    $c + r \T \subseteq   \cup_{j=0}^{k-1} \zeta_{\xi}(X_{j}^{+}) $. 
    That is, for every $k$, the space  
     $\cup_{j=0}^{k} \zeta_{\xi}(X_{k}^{+}) $ is a union of 
    circles. To get the next one,
     one replaces each circle, $c + r\T$ in the current one, 
     by $c + r  \zeta_{\xi}(X_{0}^{+} \cup X_{1}^{+}) $.
     
     Finally, $\zeta_{\xi}(X^{+}_{\xi})$ is the closure of the union
     of  $ \zeta_{\xi}(X_{k}^{+}) $, over all $k \geq 0$.
     
     We leave it as an exercise to check that, in this case, $\zeta_{\xi}$
     alone 
     is injective so $X_{\xi}^{+}$ can be embedded in the plane. 
     
     It appears that this set is an example of a fractal with 
     condensation set (see pages 91-94 of \cite{Ba:FE}). Probably a caution
     is in order: the embedding to the plane 
     we have given depends on several 
     parameters which were chosen rather arbitrarily. The underlying iterated function
     system would be affected by these choices and perhaps in a rather bad way.
 \end{ex}
 
 \begin{ques}
 \label{real:120}
 In which cases can  $X_{\xi}^{+}$ be embedded in the plane?
 \end{ques}



\end{document}